\documentclass[11pt]{article}
\usepackage{geometry}
\usepackage{amsfonts}
\usepackage{amssymb}
\usepackage{graphicx}
\usepackage{amsmath}
\usepackage{amsthm}
\usepackage{enumitem}
\usepackage{tikz}
\usepackage{mathrsfs}
\usepackage{cancel}
\usepackage{todonotes}
\usepackage{cases}
\usetikzlibrary{matrix}
\usetikzlibrary{arrows,shapes}
\usepackage{cite}
\usepackage{bm}
\usepackage{hyperref}
\hypersetup{colorlinks=true, urlcolor=blue}
\expandafter\let\expandafter\oldproof\csname\string\proof\endcsname
\let\oldendproof\endproof
\renewenvironment{proof}[1][\proofname]{%
  \oldproof[\ttfamily \scshape \bf #1. ]%
}{\oldendproof}

\def\O{{\cal O}}
\def\C{{\cal C}}

\def\A{{\cal A}}

\def\s{{ \Upsilon}}
\def\B{\mathbb{B}}
\def\R{{\bf R}}
\def\oR{\overline{\R}}
\def\N{{\rm I\!N}}

\def\ox{{\bar{x}}}
\def\tx{\tilde{x}}
\def\ty{\tilde{y}}
\def\oy{\bar{y}}
\def\oz{\bar{z}}
\def\ov{\bar{v}}
 
\def\ou{\bar{u}}
\def\X{{\bf X}}
\def\Y{{\bf Y}}
\def\Z{{\bf Z}}

\def\H{{\bf H}}
\def\S{\bf {S}}
\def\P{\bf {P}}

\def \b{{\}_{k\in\N}}}
\def \c{{\}_{k\ge 0}}}
\def\L{{\mathscr{L}}}
\def\D{{\mathscr{D}}}

\def\ykk{y^{k+1}}

\def\what{\widehat}

\def\emp{\emptyset}

\def\tto{\rightrightarrows}

\def\prox{{\mbox{prox}\,}}

\def\tto{\rightrightarrows}

\def\sub{\partial}
\def\vt{\varthetal}
\def\Hat{\widehat}
\def\ra{\rangle}
\def\la{\langle}
\def\ve{\varepsilon}
\def\omu{\bar{\mu}}

\def\olm{{\bar\lambda}}

\def\gg{\gamma}
 \def\para{{\rm par}\,}
\def\dd{\delta}
\def\al{\alpha}
\def\Th{\Theta}

\def\vt{\vartheta}

\def\ph{\varphi}

\def\toset_#1{\xrightarrow{#1}}
\DeclareMathOperator*{\mini}{minimize\;}
\DeclareMathOperator*{\argmin}{arg\, min}
\def\d{{\rm d}}
\def\dist{{\rm dist}}

\def\ri{{\rm ri}\,}

\def\inte{{\rm int}\,}
\def\gph{{\rm gph}\,}
\def\epi{{\rm epi}\,}

\def\dom{{\rm dom}\,}
\def\bd{{\rm bd}\,}
\def\ker{{\rm ker}\,}

\def\cl{{\rm cl}\,}

\def\sm{\hbox{${1\over 2}$}}
\def\rsm{\hbox{${1\over 2r}$}}

\def\rN{{\what{N}}}

\def\x1k{{x^{k-1}}}
\def\xk{{x^k}}
\def\xkk{{x^{k+1}}}
\def\l1k{{\lambda^{k-1}}}
\def\lmk{{\lambda^k}}

\def\sig1k{{\sigma_{k-1}}}

\def\ro1k{{\rho_{k-1}}}
\def\rok{{\rho_k}}

\def\bek{{\beta^k}}
\def\alk{{\alpha^k}}
\def\xik{{\beta_k}}
\def\e1k{{\epsilon_{k-1}}}
\def\ek{{\epsilon_k}}

\def\z{{s}}

\def\verl{ \;\rule[-0.4mm]{0.2mm}{0.27cm}\;}

\def\menv{(e_{1^{}/{\rho}}g)} 
\def\menvk{(e_{1^{}/{\rho_k}}g)}
\def\prox{{\rm{prox}}_{\rho^{-1}{g}}} 
\def\proxk{{\rm{prox}}_{\rho_k^{-1}g}}

\begin{document}
\begin{center}
{\Large \bf Convergence of Augmented Lagrangian Methods for Composite Optimization Problems}\\[2ex]
NGUYEN T. V. HANG\footnote{School of Physical and Mathematical Sciences, Nanyang Technological University, Singapore 639798 and Institute of Mathematics, Vietnam Academy of Science and Technology, Hanoi, Vietnam (thivanhang.nguyen@ntu.edu.sg). Research of this    author is partially supported by Singapore National Academy of Science under the grant RIE2025 NRF International Partnership Funding Initiative.} and   EBRAHIM SARABI\footnote{Corresponding author, Department of Mathematics, Miami University, Oxford, OH 45065, USA (sarabim@miamioh.edu). Research of this    author is partially supported by the U.S. National Science Foundation  under the grant DMS 2108546.}
\end{center}
\vspace*{0.05in}

\small{\bf Abstract.} Local convergence analysis of the augmented Lagrangian method (ALM) is established for a large class of composite optimization problems with nonunique Lagrange multipliers
under a second-order sufficient condition. We present a new second-order variational property, called the semi-stability of second subderivatives, and demonstrate that it is widely satisfied for 
numerous classes of functions, important for applications in constrained and composite optimization problems. Using the latter condition and a certain second-order sufficient condition, 
we are able to establish Q-linear convergence of the primal-dual sequence for an inexact version of the ALM for composite programs. 
 \\[1ex]
{\bf Keywords.}  augmented Lagrangian, nonunique Lagrange multiplier, second-order sufficient condition, ${\cal C}^2$-decomposable functions, Q-linear convergence. \\[1ex]
{\bf Mathematics Subject Classification (2000)} 90C31, 65K99, 49J52, 49J53

\newtheorem{Theorem}{Theorem}[section]
\newtheorem{Proposition}[Theorem]{Proposition}
\newtheorem{Lemma}[Theorem]{Lemma}
\newtheorem{Corollary}[Theorem]{Corollary}

\theoremstyle{definition}
\newtheorem{Definition}[Theorem]{Definition}
\newtheorem{Algorithm}[Theorem]{Algorithm}
\newtheorem{Example}[Theorem]{Example}
\newtheorem{Remark}[Theorem]{Remark}

\numberwithin{equation}{section}

\renewcommand{\thefootnote}{\fnsymbol{footnote}}

\normalsize

\section{Introduction}
The augmented Lagrangian method, also known as the method of multipliers, was introduced  and studied independently by Hestenes in \cite{Hes69} and Powell in \cite{Pow69} for nonlinear programming problems (NLPs)  with equality constraints. 
It was later extended in \cite{Roc73a} for problems of convex programming with inequality constraints by Rockafellar, where its convergence analysis was carried out by applying the proximal point algorithm to the
dual of the augmented Lagrangian problem and ensuring the Q-linear convergence of the dual sequence  and R-linear convergence of the primal sequence in the ALM. Most of the publications about the ALM afterwards exploited  a different approach  to establish the local 
convergence of this method, since the original idea in   \cite{Roc73a} relied heavily on duality, which was not available for nonconvex settings. The pioneering work of Fern\'andez and  Solodov in \cite{fs12}
 was perhaps the culmination of those efforts over the last three decades to find conditions under which the Q-linear convergence of the primal-dual sequence, constructed by the ALM, can be achieved for NLPs with nonunique Lagrange 
multipliers. Instead of looking for 
duality in the augmented Lagrangian problem, the authors in \cite{fs12} showed that  the primal-dual iterates of the ALM are indeed a solution to a particular perturbation of the KKT system of the original problem.
That paved the path for them to utilize a general framework by Fischer in \cite{Fis02} for solving generalized equations and demonstrate that the classical second-order sufficient condition alone suffices for local convergence analysis of the ALM
and no constraint qualification is necessary for such a result.

Another possible approach to conduct local convergence analysis of the ALM for nonconvex optimization problems was   recently developed by Rockafellar in \cite{r22}, where he extended his original idea 
in \cite{Roc73a} of using duality and the proximal point algorithm  to obtain convergence of the dual sequence in the ALM. The principal idea therein was to assume   the  strong variational convexity (see page~166 in \cite{r23}) of the augmented 
Lagrangian function, which was characterized in \cite[Theorem~5]{r23} to be equivalent to the {\em strong} second-order sufficient conditions for NLPs. While the approach in \cite{r22} provides a general framework 
for the local convergence analysis of the ALM, it operates under a rather strong assumption, namely the strong second-order
sufficient condition for constrained optimization problems. This, in particular, can be a challenging assumption to disentangle when we deal with important classes of optimization problems such as eigenvalue optimization problems.  
 
Following the approach by Fern\'andez and  Solodov in \cite{fs12}, we aim to analyze the local convergence of the ALM   for a class of composite optimization problems that have a representation of the form 
\begin{equation}\label{comp}
\begin{cases}
\mbox{minimize}\;\;\psi(x)\quad \mbox{subject to}\;\; x\in \Th, \\
\mbox{with}\;\; \psi(x):=\ph(x)+ g(\Phi(x)),
\end{cases}
\end{equation} 
where  $\ph:\X\to \R$ and $\Phi:\X\to \Y$ are twice continuously differentiable functions, $g: \Y \to \oR:=[-\infty,\infty]$ is a proper lsc convex function, $\Th$ is a   polyhedral convex  set, and $\X$ and $\Y$ are  finite dimensional Hilbert spaces. 
The composite problem in \eqref{comp} encompasses many important classes of constrained and composite optimization problems such as
NLPs and nonlinear semidefinite programming problems (SDPs), convex piecewise linear-quadratic composite optimization problems, and eigenvalue related optimization problems.
We reveal two major second-order variational conditions for the convex function $g$ in \eqref{comp}, which together with  the second-order sufficient condition allow us to establish Q-linear convergence of the primal-dual sequence of the ALM for \eqref{comp}. In particular, we do not assume any constraint qualification, and hence  are able to deal with  composite problems with nonunique Lagrange multipliers. 

The main idea of the ALM  is to smooth out the nondifferentiable parts in the composite function $\psi$ in \eqref{comp} and to minimize then the obtained augmented function  for the next iterate of the method. 
To elaborate more, consider the   augmented Lagrangian function 
 $\L:\X\times \Y\times (0,\infty)\to \R$,   defined  by 
\begin{equation*}\label{aug}
\L(x,y,\rho):=\inf_{u\in \Y}\big\{\psi(x,u)+\frac{1}{2\rho} \|u\|^2-\la y,u\ra\big\},\quad (x,y,\rho)\in \X\times \Y\times (0,\infty),
\end{equation*}
where $\psi(x,u)$ is a {\em partial} perturbation of $\psi$ in \eqref{comp}, given by  $\psi(x,u):=\ph(x)+ g(\Phi(x)+u)$  for any $(x,u)\in \X\times \Y$.
Note that the full perturbation of $\psi$ requires to replace $\psi$ with $\psi+\dd_\Th$ and consider then  a second perturbation variable for $\Th$. This complicates our convergence analysis and  seems unnecessary, however. 
A direct calculation then shows that for any $(x,y,\rho)\in  \X\times \Y\times (0,\infty)$, the augment Lagrangian $\L$ can be equivalently expressed as
\begin{equation}\label{aug}
\L(x,y,\rho)=\ph(x)+e_{ {1}/{\rho}} g \big(\Phi(x)+\rho^{-1} y\big)-\sm\rho^{-1}\| y \|^2,
\end{equation}
where $e_{1/\rho} g$ stands for  the Moreau envelope of   $g$, given  by 
\begin{equation*}
\menv(y):=\inf_{z\in \Y}\big\{g(z)+\tfrac{1}{2}\rho\|y-z\|^2\big\}, \quad y\in \Y.
\end{equation*}
Given the current triple $(\xk, y^k,\rok)\in \X\times \Y\times (0,\infty)$, the {\em exact} version of the ALM generates the next primal iterate $\xkk$ and the dual iterate $y^{k+1}$, respectively,  by 
\begin{equation}\label{augxy}
\xkk \in \argmin_{x\in \Th} \L(x,y^k,\rok)\quad \mbox{and}\quad y^{k+1}=y^k+\rok\nabla_y \L(\xkk,y^k,\rok). 
\end{equation}
Since selecting $\xkk$ as an exact minimizer of  the augmented Lagrangian function does not seem practical, we are going to consider a more realistic scenario 
and demand that   the next primal iterate $\xkk$ be an {\em approximate} stationary solution to the constrained augmented problem 
\begin{equation}\label{subp}
\mbox{minimize}\;\; \L(x,y^k,\rok)\quad \mbox{subject to}\quad x\in \Th,
\end{equation}
 namely it satisfies the condition 
\begin{equation}\label{xkk}
\dist\big(-\nabla_x\L(\xkk, y^k, \rok),N_\Th(\xkk)\big)\leq \ek,
\end{equation} 
where $\ek\geq 0$ is called the {\em tolerance} parameter.  Note that an exact minimizer $\xkk$ as the one in \eqref{augxy} satisfies \eqref{xkk} with $\ek=0$. 

As pointed earlier, the  local convergence of the ALM for \eqref{comp} was conducted  in \cite{fs12}  for NLPs when the second-order sufficient condition is satisfied.
Using the tools of second-order variational analysis, we showed in \cite{HaS21} that a similar  convergence analysis  can be achieved for \eqref{comp} when $g$ therein belongs to the class
of convex piecewise linear-quadratic functions. For SDP, the Q-linear convergence of the primal-dual sequence for the exact ALM was justified in \cite{ssz} when the nondegeneracy 
condition and the strong second-order sufficient condition are satisfied. In the same setting, 
it was recently shown in \cite[Theorem~2]{ding} that if a relative interior condition for the underlying Lagrange multiplier and the SOSC are satisfied, the inexact version 
of the ALM is Q-linearly convergent. Using a stronger version of the SOSC, the authors in \cite{ding2} obtained the R-linear convergence for   the primal sequence and Q-linear convergence for 
the dual sequence of the ALM without assuming the strict complementarity condition. We should also mention that a modified version of the ALM, called the {\em safeguarded}  ALM 
 in which the $y^k$  in  \eqref{subp} is replaced with a certain vector $w^k$ chosen from a bounded set, 
 was studied in \cite{abm1,abm2, ks}, where it was shown that 
 this modification has  an interesting  global convergence. The main motivation for this modification comes from the fact that  the dual sequence, constructed by the  ALM, may be unbounded in general.  
 
In our recent work \cite{HaS21}, we extended the local convergence analysis of the ALM in \cite{fs12} for a class of composite optimization 
problems that the modeling function therein,   the function $g$ in \eqref{comp}, was assumed to be convex piecewise linear-quadratic (CPLQ). In the present paper, our primary goal is twofold. First, we are going  to demonstrate that 
a similar result can be obtained for composite optimization problem in \eqref{comp} with the modeling function $g$ belonging  to a large class of functions, called ${\cal C}^2$-decomposable; see \eqref{decom}. Second, we provide two major conditions
under which the  local convergence of the ALM for \eqref{comp} can be achieved. In doing so, we introduce two new concepts, the semi-strict graphical derivative and semi-strict second subderivative, and calculate them 
for some important classes of functions. We also introduce a new concept, called semi-stability of second subderivatives, and demonstrate that it holds for various important classes of functions in optimization. This property is then utilized 
to ensure a uniform version of the quadratic growth condition for the augmented Lagrangian function in \eqref{aug}, which is a major tool in our approach to analyze the convergence of the ALM. 

The outline of the paper is as follows. We begin in Section~\ref{pre} by recalling an iterative framework for our  inexact ALM and reviewing our notation. 
The goal of Section~\ref{calm} is to provide a general property under which we can ensure the metric subregualrity of the Karush-Kuhn-Tucker (KKT) system associated with \eqref{comp}. 
In Section~\ref{sssub}, we introduce the concept of semi-stability of second subderivatives and use it to justify  the quadratic growth condition for the augmented Lagrangian function.
Section~\ref{local} is devoted to establishing the Q-linear convergence of the primal-dual sequence generated by the proposed inexact ALM. 
Finally, in Section~\ref{appe}, we present
two independent results that are important for our developments in this paper.

\section{Prelimenaries}\label{pre}
\subsection{Notation}
In what follows,  suppose that $\X$ and $\Y$ are finite dimensional Hilbert spaces.
We denote by $\B$ the closed unit ball in the space in question and by $\B_r(x):=x+r\B$ the closed ball centered at $x$ with radius $r>0$. 
 In the  product space $\X\times \Y$, we use the norm $\|(w,u)\|=\sqrt{\|w\|^2+\|u\|^2}$ for any $(w,u)\in \X\times \Y$.
 Given a nonempty set $C\subset\X$, the symbols $\inte C$,  $\ri C$, $C^*$,  and $\para C$ signify its interior, relative interior, polar cone, and the   linear subspace parallel to the affine hull of $C$, respectively. 
 For any set $C$ in $\X$, its indicator function is defined by $\dd_C(x)=0$ for $x\in C$ and $\dd_C(x)=\infty$ otherwise. We denote
 by $P_C$ the projection mapping onto $C$ and  by $\dist(x,C)$  the distance between $x\in \X$ and a set $C$.
 For a vector $w\in \X$, the subspace $\{tw |\, t\in \R\}$ is denoted by $[w]$. The set of nonnegative number is denoted by $\R_+$.

Let $\{C^t\}_{t>0}$ be a parameterized family of sets in $\X$. Its outer limit is defined  by 
\begin{equation*}
\limsup_{t\searrow 0} C^t= \big\{x\in \X|\; \exists \; t_k \searrow 0 \;\exists\;   \; x^{t_k}\to x \;\;\mbox{with}\;\; x^{t_k}\in C^{t_k}\big\}.
\end{equation*}
 Given a nonempty set $\Omega\subset\X$ with $\ox\in \Omega$, the {  tangent cone}  to $\Omega$ at $\ox$, denoted $T_\Omega(\ox)$,  is defined  by
$
T_\Omega(\ox) = \limsup_{t\searrow 0} \frac{\Omega - \ox}{t}.
$
The  regular/Fr\'{e}chet normal cone $\rN_\Omega(\ox)$ to $\Omega$ at $\ox$ is defined by
 $\rN_\Omega(\ox) = T_\Omega(\ox)^*$. For $x\notin \Omega$, we set $\rN_\Omega(x) = \emptyset$. The limiting/Mordukhovich normal cone $N_\Omega(\ox)$ to $\Omega$ at $\ox$ is
 the set of all vectors $\ov\in \X$ for which there exist sequences  $\{x^k\b$ and  $\{v^k\b$ with $v^k\in \rN_\Omega( \xk)$ such that 
$(x^k,v^k)\to (\ox,\ov)$. When $\Omega$ is convex, both normal cones boil down to that of convex analysis.  
Given a function $f:\X \to \oR$ and a point $\ox\in\X$ with $f(\ox)$ finite, the subderivative function $\d f(\ox)\colon\X\to\oR$ is defined by
\begin{equation*}\label{fsud}
\d f(\ox)(w)=\liminf_{\substack{
   t\searrow 0 \\
  w'\to w
  }} \frac{f(\ox+tw')-f(\ox)}{t}.
\end{equation*}
A vector $v\in \X$ is called a regular subgradient of $f$ at $\ox$ if $(v,-1)\in \Hat N_{\epi f}(\ox,f(\ox))$ with $\epi f=\{(x,\al)\in \X\times \R\, |\, f(x)\le \al\}$ being the epigraph of $f$. The set of all regular subgradients of $f$ at $\ox$
is    denoted by $\Hat \sub f(\ox)$. Similarly, we can define $\sub f(\ox)$ using the limiting normal cone $N_{\epi f}(\ox,f(\ox))$. When $f$ is a convex function, both sets reduce to the well-known subdifferential of convex functions.  
The critical cone of $f$ at $\ox$ for $\bar v$ with $\bar v\in   \sub f(\ox)$ is defined by 
\begin{equation*}\label{cricone}
{K_f}(\ox,\bar v)=\big\{w\in \X\,\big|\,\la\bar v,w\ra=\d f(\ox)(w)\big\}.
\end{equation*}
When $f=\dd_\Omega$, where $\Omega$ is a nonempty subset of $\X$, the critical cone of $\dd_\Omega$ at $\ox$ for $\ov$ is denoted by $K_\Omega(\ox,\ov)$. In this case, the above definition of the critical cone of a function 
boils down to  the well-known concept of the critical cone of a set (see \cite[page~109]{dr}), namely $K_\Omega(\ox,\ov)=T_\Omega(\ox)\cap [\ov]^\perp$. 
The {second subderivative} of $f$ at $\ox$ for $\ov \in   \sub f(\ox)$, denoted $\d^2 f(\bar x , \ov)$, is an extended-real-valued function defined  by 
\begin{equation*}\label{ssd}
\d^2 f(\bar x , \ov)(w)= \liminf_{\substack{
   t\searrow 0 \\
  w'\to w
  }} \Delta_t^2 f(\ox , \ov)(w'),\;\; w\in \X,
\end{equation*}
where $\Delta_t^2 f(\ox , \ov)$  is the parametric  family of 
second-order difference quotients of $f$ at $\ox$ for $\ov$ and is defined for any $w\in \X$ and $t>0$ by 
\begin{equation}\label{sodq}
\Delta_t^2 f(\bar x , \ov)(w)=\frac{f(\ox+tw)-f(\ox)-t\langle \ov,\,w\rangle}{\frac {1}{2}t^2}.
\end{equation}
The function $f$ is called twice epi-differentiable at $\ox$ for $\ov$ if for any $w\in \X$ and any $ t_k\searrow 0$, there exists $w^k\to w$ such that $\Delta_{t_k}^2 f(\bar x , \ov)(w^k)\to \d^2 f(\bar x , \ov)(w)$;
see \cite[Definition~13.6(b)]{rw}. 
  For  a set-valued mapping $F: \X \rightrightarrows \Y$, its graph and domain are defined, respectively,  by  
$
\gph F=\big\{(x,y)\in \X\times \Y\, \big|\, y\in F(x)\big\} 
$
and $ \dom F= \big\{ x\in \X\, \big|\, F(x)\neq \emptyset\big\}.$
We say that $F$ is calm at $x$ for $y\in F(x)$ if there are neighborhoods $U$ of $x$ and $V$ of $y$ and a positive constant $\kappa$ for which we have 
$$
F(x')\cap V\subset F(x)+\kappa\,\|x-x'\|\B\quad \mbox{for all}\;\; x'\in U.
$$
The mapping $F$ is called metrically subregular at $x$ for $y\in F(x)$ if  there are a neighborhood $U$ of $x$  and a positive constant $\ell$
such that $\dist(x',F^{-1}(y))\le \ell\, \dist (y,F(x'))$ for any $x'\in U$. 

\subsection{An Iterative Framework for ALM}

 As shown later in this section, the primal-dual sequence $\{(\xk,y^k)\b$, generated by the ALM considered in this paper,  are solutions to   a particular perturbation of 
the KKT system associated with the composite optimization problem in \eqref{comp}, which is   
  given by 
\begin{equation}\label{vs} 
  0\in\nabla_x L(x,y) + N_\Th(x),\;\;y\in \sub g(\Phi(x)),
\end{equation}
where $L(x,y)= \ph(x)+\la y, \Phi(x)\ra$ with $(x,y)\in \X\times \Y$ is the Lagrangian of \eqref{comp} and  
where the functions $g$, $\Phi$, and the set $\Th$ are taken from \eqref{comp}. It is easy to see that 
\eqref{vs} can be written as the generalized equation 
\begin{equation}\label{geco} 
\begin{bmatrix}
0\\
0
\end{bmatrix}\in  \begin{bmatrix}
\nabla_xL(x, y)\\
-\Phi(x)
\end{bmatrix}
+
 \begin{bmatrix}
N_\Th(x)\\ (\partial g)^{-1}(y)
\end{bmatrix}.
\end{equation}
To elaborate more on the kind of perturbation we should consider for this generalized equation,  we briefly recall an abstract iterative framework, suggested by Fischer in \cite{Fis02}, for solving generalized equations with nonisolated solutions.
Given the mappings $\Psi: \H\to \H'$ and $G:\H \tto\H'$, where $\H$ and $\H'$ are finite dimensional Hilbert spaces,  consider the generalized equation formulated by
\begin{equation}\label{ge}
0\in \Psi(u) + G(u).
\end{equation}
Define the  solution mapping $\s:\H'\tto \H$ to the canonical perturbation of \eqref{ge} by    
\begin{equation}\label{solm}
\s(v)=\big\{u\in \H|\, v \in \Psi(u) +G(u)\big\}, \;\;v\in \H'.
\end{equation}
Given  a set of parameters $\P$, iterative algorithms for solving \eqref{ge} often generate a sequence $\{u^k\b$ by iteratively solving subproblems of the form
\begin{equation}\label{age}
0\in \A(u, u^k, p_k)+G(u),
\end{equation}
in which the single-valued part in \eqref{ge} is replaced with a set-valued mapping $\A:\H\times \H\times \P\tto \H'$,  which gives an approximation of $\Psi$ around the current iterate, while the set-valued part is kept unchanged. 
In order to ensure the convergence and to establish the rate of convergence of $\{u^k\b$, we have to  choose a solution to \eqref{age} satisfying certain properties. Namely, 
given the current iterate $u^k$ and a parameter $p_k$, we choose the next iterate $u^{k+1}$ sufficiently close to $u^k$ satisfying
\begin{equation}\label{ukk}
u^{k+1} \in \big\{ u\in \H\, \big|\, 0\in \A(u, u^k, p_k)+ G(u)\; \textrm{ and }\; \|u-u^k\|\leq c\, \dist(u^k,    \s(0))\big\},
\end{equation}
where $ c>0$ is arbitrary yet fixed, and $\s(0)$, taken from \eqref{solm} with $v=0$, is the set of solutions to \eqref{ge}.   Below, we recall a result, established in  \cite[Theorem~4.1]{IzK13} (see also  \cite[Theorem~7.13]{is14}), in which 
the local convergence  analysis of the sequence $\{u^k\b$ was obtained under rather mild assumptions. We should point out that assumption (a) in the theorem below was expressed in  \cite[Theorem~4.1]{IzK13} in a slightly stronger sense, namely 
 the solution mapping $\s$ was assumed to be upper Lipschitzian instead of being calm. It is not hard to see from its proof, however, that calmness suffices for this result. The nonparametric version of this result was later appeared in  \cite[Theorem~7.13]{is14} with upper Lipschitzian replaced by calmness. 

\begin{Theorem}\label{thm:fischer} Let $\ou \in \s(0)$, where $\s$ is the solution mapping from \eqref{solm}. Assume that $\s(0)$ is locally closed around $\ou$ and that the following properties   hold for some positive constant $c$:
\begin{enumerate}[noitemsep,topsep=2pt]
\item {\rm(}calmness of solution mapping{\rm)} the solution mapping $\s$ is calm at $\ou$ for $0\in \H'$ with constant   $\ell_1>0$;

\item {\rm(}solvability of subproblems{\rm)} there exists a positive constant $\varepsilon_1$ such that for any $\tilde u \in \B_{\varepsilon_1}(\ou)$ and any $p\in \P$ the localized solution set
\begin{equation*}
\big\{ u\in \X\, \big|\, 0\in \A(u, \tilde u, p)+ G(u)\; \textrm{ and }\; \|u-\tilde u\|\leq c\, \dist(\tilde u, \s(0))\big\}
\end{equation*}
is nonempty;

\item {\rm(}precision of approximation{\rm)} there exist a positive constant $\varepsilon_2$ and a function $\omega: \H\times\H\times \P:\to \R_+$ such that 
\begin{equation*}
 \sup\big\{\omega(u, \tilde u, p)\, \big|\, \tilde u \in \B_{\varepsilon_2}(\ou),\, \| u - \tilde u\|\leq  c\, \dist(\tilde u, \s(0)), \, p\in {\P}\big\} < 1/\ell_1,
\end{equation*}
where constant $\ell_1>0$ is taken from {\rm(a)}, and the estimate
\begin{equation*}
\sup\big\{\|w\|\, \big|\, w \in \Psi(u) - \A(u, \tilde u, p)\big\} \leq \omega(u, \tilde u, p)\, \dist(\tilde u,\s(0))
\end{equation*}
holds for all $\tilde u\in \B_{\varepsilon_2}(\ou)$, all $u\in \H$ with $\|u-\tilde u\|\leq  c\, \dist(\tilde u, \s(0))$, and all $p\in \P$.
\end{enumerate}
Then there exists $\varepsilon_0>0$ such that for any starting point $u^0$ chosen from $\B_{\varepsilon_0}(\ou)$ and any sequence $\{p_k\}_{k\in \N}\subset \P$, 
the iterative scheme in \eqref{ukk} 
generates a sequence $\{u^k\}_{k\in \N}$ converging to some $\hat u \in \s(0)$.
The rates of convergence of $\{u^{k}\b$ to $\hat u$ and of $\{\dist(u^k, \s(0))\b$ to $0$ are Q-linear, and they are Q-superlinear provided that $\omega(u^{k+1}, u^k, p_k)\to 0$ as $k\to\infty$.
Moreover, for any $\varepsilon>0$, we have  $\|\hat u -\ou\|<\ve$ if   $u^0$ is chosen sufficiently close to $\ou$.
\end{Theorem}

Our main objective  in this paper is to show that if a second-order sufficient condition for optimality (see \eqref{sosc1}) holds for the composite optimization problem in \eqref{comp}  
and a certain calmness of the multiplier mapping associated with this problem (see \eqref{calmm}) is satisfied, then all   assumptions (a)-(c) in Theorem~\ref{thm:fischer}
can be verified for many important classes of constrained and composite optimization problems. 

Our final goal in this section is to show that  the primal-dual sequence $\{(\xk,y^k)\b$, generated by the {\em inexact} ALM in this paper, can be fit into the iterative pattern described above.
To this end, recall first that  the proximal mapping of  a convex function $g:\Y\to \oR$ is defined by 
\begin{equation*}\label{pr}
{\rm{prox}}_{rg}(y)=\underset{z\in \Y}{\argmin} \big\{g(z)+\rsm\|y-z\|^2\big\},\quad y\in\Y,
\end{equation*}
where $r$ is a positive constant. 
In what follows, when $r=1$, the proximal mapping of $g$ will be denoted by ${\rm prox}_g$. It  follows from   \cite[Proposition~12.19]{rw} and \cite[Theorem~2.26]{rw}, respectively, 
 that for any $y\in \Y$ we always have 
\begin{equation}\label{prm}
{\rm{prox}}_{rg}(y)=\big(I+r\sub g\big)^{-1}(y) \quad \mbox{and}\quad \nabla e_rg(y)=\big(rI+(\sub g)^{-1}\big)^{-1}(y),
\end{equation}
where $I$ stands for the identity mapping from $\Y$ onto $\Y$. Using these relationships, one can equivalently reformulate the dual update $\ykk$ in \eqref{augxy}
as
\begin{equation}\label{kkt2}
\ykk \in \partial g\big(\Phi(\xkk) - \rho_k^{-1}(\ykk-y^k)\big)\quad \mbox{or}\quad \ykk=\nabla\menvk\big(\Phi(\xkk)+ \rho_k^{-1}{y^k}\big).
\end{equation}
To express the KKT system in \eqref{vs}  in the form of the generalized equation in \eqref{ge}, we define the mappings   $\Psi: \X\times \Y \to \X\times\Y$ and  $G: \X\times \Y \tto \X\times\Y$  by
\begin{equation}\label{ALMge}
\Psi(x, y) = \begin{bmatrix}
\nabla_xL(x, y)\\
-\Phi(x)
\end{bmatrix}
\quad\textrm{ and }\quad
G(x, y) =\begin{bmatrix}
N_\Th(x)\\ (\partial g)^{-1}(y)
\end{bmatrix},
\end{equation}
which clearly demonstrates that \eqref{vs} can be covered by \eqref{ge}. 
Moreover, the inexact primal update $\xkk$, satisfying  \eqref{xkk},      can be equivalently described via \eqref{kkt2} by  
\begin{align*}
0 &\in \nabla_x\L(\xkk, y^k, \rok) +\ek\B+N_\Th(\xkk)\nonumber\\
&= \nabla \varphi(\xkk) +\nabla\Phi(\xkk)^*\nabla\menvk\big(\Phi(\xkk)+ \rho_k^{-1}{y^k}\big)+\ek\B+N_\Th(\xkk)\nonumber\\
&=\nabla_x L(\xkk, \ykk) +\ek\B+N_\Th(\xkk).
\end{align*}
In order to show that the primal-dual iterate $(\xkk,\ykk)$ satisfies \eqref{age}, we  define the mapping $\A: \X\times \Y\times \X\times \Y\times \R_+\times (0, \infty) \tto \X\times \Y$ by
\begin{equation}\label{A}
\A(x, y, \tx, \ty, \epsilon, \rho) := \begin{bmatrix}
\nabla_xL(x, y) + \epsilon \B\\
-\Phi(x) +\rho^{-1}(y - \ty)
\end{bmatrix}
\end{equation}
for a given quadruple  $(\tx, \tilde y,\epsilon, \rho)\in \X\times \Y\times \R_+\times (0,\infty)$.  Combining these with the first inclusion in \eqref{kkt2}
tells us that the primal-dual iterate $(\xkk,\ykk)$ is a solution to the subproblem 
$$
(0,0)\in \A(x, y, \xk, y^k, \epsilon_k, \rok)+ G(x, y). 
$$
We close this section by recording some properties of the augmented Lagrangian function in \eqref{aug}, which directly result from \eqref{prm}.
Note that for any $(x,\rho)\in \X\times (0,\infty)$, the mapping $y\mapsto \L(x,y,\rho)$ is ${\cal C}^1$ and its gradient can be calculated via  \eqref{prm} as
\begin{equation}\label{ueq19}
\nabla_y\L(x, y, \rho)=\rho^{-1}\big(\nabla\menv\big(\Phi(x)+\rho^{-1}{y}\big) - y\big)
= \Phi(x)-\prox \big(\Phi(x)+ \rho^{-1}y\big).
\end{equation} 

\begin{Proposition}\label{fopag}
Let $(\ox,\oy)$ be a solution to the KKT system in \eqref{vs}. Then for any $\rho>0$,  the following properties are satisfied.
\begin{enumerate}[noitemsep,topsep=0pt]
\item $\L(\ox,\oy,\rho)=\psi(\ox)=\ph(\ox)+g\big(\Phi(\ox)\big)$.
\item $\nabla_x\L(\ox,\oy,\rho) = \nabla_x L(\ox,\oy)$ and $\nabla_y\L(\ox,\oy,\rho)=0$.
\end{enumerate}
\end{Proposition}

\section{Metric Subregularity of KKT Systems}\label{calm}

As shown in Theorem~\ref{thm:fischer}, the calmness of  the solution mapping in \eqref{solm} 
is required for local convergence analysis of the ALM for the composite optimization problem \eqref{comp}. In this section, we aim to provide a sufficient condition for 
an error bound estimate for  the generalized equation in \eqref{geco}, which will be shown 
in Proposition~\ref{casol} that is  equivalent to the aforementioned calmness of \eqref{solm}. 
To achieve our goal, recall that the graphical derivative of a set-valued mapping $F:\X\tto\Y$ at $\ox$ for $\oy\in F(\ox)$, denoted by $DF(\ox, \oy)$, is a set-valued mapping from $\X$ into $\Y$, defined by
 $$
 \gph DF(\ox, \oy)= T_{\gph F}(\ox,\oy).
 $$
Employing the definition of the tangent cone, one can easily obtain an equivalent sequential description of $\gph DF(\ox, \oy)$. 
If, in addition, for any $\eta\in DF(\ox, \oy)(w)$ and any choice of $t_k\searrow0$,  there exist sequences $w^k \to w$ and $\eta^k \to \eta$ with $\oy+t_k\eta^k \in F(\ox +t_kw^k)$, 
 then $F$ is said to be proto-differentiable at $\ox$ for $\oy$; see  \cite[page~331]{rw} for more details. One of the main source of proto-differentiability in variational 
 analysis is  subgradient mappings of various classes of functions including CPLQ functions; see \cite{mms,ms20,rw} for more examples.

\begin{Definition}[semi-strict graphical derivatives]\label{ssgd}
The semi-strict graphical derivative  of $F:\X\tto \Y$  at $\ox$ for $\oy\in F(\ox)$, denoted by $\Hat DF(\ox, \oy)$, is a set-valued mapping from $\X$ to $\Y$ defined by
$$
\eta\in \Hat DF(\ox, \oy)(w)
\iff \begin{cases}
\exists\; t_k\searrow 0, \, (x^k,y^k)\to (\ox,\oy),\; \widehat y^k\to \oy\;\;\mbox{with}\; y^k\in F(x^k), \; \widehat y^k\in F(\ox)\\
 \mbox{such that}\; (\frac{x^k-\ox}{t_k},\frac{y^k-\widehat y^k}{t_k})\to (w,\eta).
\end{cases}
$$
\end{Definition}
Using the definition of the outer limit of a family of sets, one can see that the semi-strict graphical derivative of $F$ can be equivalently expressed as
\begin{equation}\label{ulss}
 \Hat DF(\ox, \oy)(w)= \limsup_{\substack{
t\searrow 0, \, w'\to w \\
y \to \oy,\; y\in F(\ox)}} 
\frac{F(\ox+tw') - y}{t}, \quad w\in \X.
\end{equation}
By definition, we can immediately conclude that $DF(\ox, \oy)(w)\subset  \Hat DF(\ox, \oy)(w)$ for any $w\in \X$. Our interest in the semi-strict graphical 
derivative resides in Theorem~\ref{error} in which we show using this notion that the second-order sufficient condition in \eqref{sosc1}   ensures an error bound for   
 the KKT system  of \eqref{comp}. In order to do so, we should investigate how $DF(\ox, \oy)$ and $\Hat DF(\ox, \oy)$ relate to each other when $F$ is   subgradient mappings.
We begin with a result that can be used as our guide to pursue such a relationship. Recall first 
from \cite[Definition~7.25]{rw} that an lsc function $f: \X\to \oR$ is said to be subdifferentially regular at $\ox\in \X$  if    $f(\ox)$ is   finite and $\rN_{\epi f}(\ox, f(\ox))= N_{\epi f}(\ox,f(\ox))$.

\begin{Proposition}\label{prop:gr}
Assume that $f: \X\to \oR$  and $(\ox, \ov)\in \gph \partial f$. If $\sub f$ is proto-differentiable at $\ox$ for $\ov$,
then we have 
\begin{equation}\label{gr1}
\cl\big(D(\partial f)(\ox, \ov)(w) - T_{\sub f(\ox)}(\ov)\big) \subset \Hat D(\partial f)(\ox, \ov)(w)\quad\textrm{ for all }\quad w\in \X. 
\end{equation}
In addition, if $f$ is subdifferentially regular at $\ox$, then $T_{\sub f(\ox)}(\ov)$ in the left-hand side of \eqref{gr1} can be replaced with $K_f(\ox, \ov)^*$.
\end{Proposition}

\begin{proof}
Observe that the set $\Hat D(\partial f)(\ox, \oy)(w)$ is   closed due to   \eqref{ulss}  and the fact that  the outer limit of a sequence of sets is always closed (cf. \cite[Prposition~4.4]{rw}). 
To prove \eqref{gr1},  it suffices to take $\eta \in D(\partial f)(\ox, \ov)(w)$ and $\zeta \in T_{\sub f(\ox)}(\ov)$, and show    that  $\eta-\zeta\in \Hat D(\partial f)(\ox, \ov)(w)$. 
By $\zeta \in T_{\partial f(\ox)}(\ov)$, we find sequences $t_k\searrow0$ and $\zeta^k \to \zeta$ with $\ov+ t_k\zeta^k \in \partial f(\ox)$. 
Moreover,   proto-differentiability of $f$ at $\ox$ for $\ov$ implies  that there exist sequences $w^k \to w$ and $\eta^k\to\eta$ with $\ov+t_k\eta^k \in \partial f(\ox+t_kw^k)$. 
Putting all of these together,  we have
$(\ov+t_k\zeta^k)+ t_k(\eta^k -\zeta^k) \in \partial f(\ox+t_kw^k) $, which leads us to 
  $\eta - \zeta \in \Hat D(\partial f)(\ox, \ov)(w)$ and hence proves \eqref{gr1}.
  
  Assume now that $f$ is subdifferentially regular at $\ox$. According to \cite[Theorem~8.30]{rw}, we know that $\partial f(\ox)$ is closed and convex and that
\begin{equation}\label{gr2}
K_f(\ox, \ov) =  N_{\partial f(\ox)}(\ov),
\end{equation}
which clearly justifies the final claim. 
\end{proof}

We next examine whether the opposite inclusion in \eqref{gr1} holds for important classes of the modeling function  $g$ in the composite optimization problem \eqref{comp}. 
Before moving any further, let us show  that if $f:\X\to \oR$ is convex and $(\ox,\ov)\in \gph \sub f$, then we always have 
\begin{equation}\label{gr3}
\dom \Hat D(\partial f)(\ox, \ov) \subset K_f(\ox, \ov).
\end{equation}
Indeed, assume that for  $w\in \X$ there exists $\eta \in \Hat D(\partial f)(\ox, \ov)(w)$. 
By definition, we can select sequences $t_k\searrow0$, $w^k\to w$, $v^k \in \partial f(\ox+t_kw^k)$, and $\Hat v^k\to \ov$ with $\Hat v^k\in \partial f(\ox)$ such that $(v^k - \Hat v^k)/t_k \to \eta$. 
In particular, we have $v^k \to \ov$ and
\begin{equation*}
\frac{f(\ox+t_kw^k) -f(\ox)}{t_k} \leq \la v^k, w^k\ra
\end{equation*}
for all $k$. 
Passing to the limit as $k\to\infty$ gives us $\d f(\ox)(w) \leq \la \ov, w\ra$ that actually holds as equality due to $\ov \in \partial f(\ox)$ and \cite[Exercise~8.4]{rw}. 
Thus, $w \in K_f(\ox, \ov)$, which proves \eqref{gr3}. 

\begin{Example}\label{ex:4}
Assume that  $f : \X \to \oR $ with $\X=\R^n$ 
is a CPLQ function. Recall  that $f$ is called  piecewise linear-quadratic if $\dom f = \cup_{i=1}^{s} C_i$ with $s\in \N$ and $C_i $ being polyhedral convex  sets for $i = 1, \ldots, s$, and if $f$ has a representation of the form
\begin{equation*} \label{PWLQ}
f(x) = \sm \langle A_i x ,x \rangle + \langle a_i ,x \rangle + \alpha_i  \quad \mbox{for all} \quad  x \in C_i,
\end{equation*}
where $A_i$ is an $n \times n$ symmetric matrix, $a_i\in \R^n$, and $\alpha_i\in \R$ for $i = 1, \ldots, s$.
Pick $\ox \in \dom f$ and $\ov \in \partial f(\ox)$. 
Since $f$ is convex, it is subdifferentially  regular at $\ox$, and $\partial f$ is proto-differentiable at $\ox$ for $\ov$ according to  \cite[Proposition~13.9]{rw} and    \cite[Theorem~13.40]{rw}. 
Proposition~\ref{prop:gr} tells us  that the inclusion in \eqref{gr1} holds. 
If $w \notin \dom \Hat D(\partial f)(\ox, \ov)$, then the opposite inclusion is trivial. 
We now verify the opposite  inclusion for any $w\in \dom \Hat D(\partial f)(\ox, \ov)$. It follows from \eqref{gr3} that $w\in K_f(\ox,\ov)$.  
Take  $\eta \in \Hat D(\partial f)(\ox, \ov)(w)$ and find, by definition, sequences $t_k\searrow0$, $w^k\to w$, $v^k \in \partial f(\ox+t_kw^k)$, and $\Hat v^k\to \ov$ with $\Hat v^k\in \partial f(\ox)$ such that $(v^k - \Hat v^k)/t_k \to \eta$. 
We then have
\begin{equation*}
(t_kw^k, v^k - \ov) = (\ox+t_kw^k, v^k) - (\ox, \ov) \in \big(\gph\partial f\big) - (\ox, \ov).
\end{equation*} 
It follows from the reduction lemma for CPLQ functions
from  \cite[Theorem~2.3]{s20} that there exists a neighborhood $\O$ of $(0, 0)\in \R^n\times\R^n$ for which we have 
\begin{equation*}
\big(\big(\gph\partial f\big) - (\ox, \ov)\big)\cap\O =\big(\gph D(\partial f)(\ox, \ov)\big)\cap\O.
\end{equation*}
We can assume for any $k$ sufficiently large that $(t_kw^k, v^k - \ov)\in \O$ and therefore obtain  $(w^k, (v^k - \ov)/t_k) \in \gph D(\partial g)(\ox, \ov)$. 
  By \cite[Proposition~2.4(b)]{s20} and $w\in K_f(\ox,\ov)$, 
  there exists a constant $\tau\geq 0$ such that for any $k$ sufficiently large we have 
\begin{equation}\label{ex4.1}
\frac{v^k-\ov}{t_k} \in D(\partial f)(\ox, \ov)(w^k) \subset D(\partial f)(\ox, \ov)(w) +\tau\|w^k-w\|\B.
\end{equation}
Note also that $\Hat v^k - \ov \in \sub f(\ox)-\ov \subset T_{\partial f(\ox)}(\ov)$. This, coupled with \eqref{gr2}, which   holds for the CPLQ function $f$,    brings us to  $(\Hat v^k - \ov)/t_k \in  K_f(\ox, \ov)^*$ for all $k$.
Using this and \eqref{ex4.1},  we find $\eta^k \in D(\partial f)(\ox, \ov)(w)$ and $\zeta^k \in \B$ such that
\begin{equation*}
\frac{v^k-\Hat v^k}{t_k}-\tau\|w^k-w\|\zeta^k=\eta^k-\frac{\Hat v^k -\ov}{t_k} \in D(\partial f)(\ox, \ov)(w) - K_f(\ox, \ov)^*
\end{equation*}
for any $k$ sufficiently large.
Passing to the limit as $k\to\infty$ and using the boundedness of $\zeta^k$, we arrive at the inclusion
\begin{equation*}
\eta \in \cl\big (D(\partial f)(\ox, \ov)(w) - K_f(\ox, \ov)^*\big),
\end{equation*}
which in turn demonstrates that   the inclusion in \eqref{gr1}  becomes equality for any $w\in \R^n$ for CPLQ functions.
Note that in this framework, we can drop the closure in \eqref{gr1}  to obtain 
$$
\Hat D(\partial f)(\ox, \ov)(w)=D(\partial f)(\ox, \ov)(w) - K_f(\ox, \ov)^*\quad \mbox{for any}\;\; w\in \X,
$$
since the right-hand side is the sum of two polyhedral convex sets,  which is known to be a polyhedral convex set, and hence is always closed. 
\end{Example}

Our next goal is to show that we should expect a similar result as the one in Example~\ref{ex:4} outside the polyhedral framework. 
We begin with proving the following result, which is of its own interest. Recall from 
 \cite{sh03} that a   function $g: \Y \to \oR$ is called ${\cal C}^2$-{\em decomposable} at $\ou\in \Y$ if $g(\ou)$ is finite and $g$  enjoys the composite representation 
 \begin{equation}\label{decom}
g(u) = g(\ou) + \vartheta(\Xi(u))\quad \textrm{ for }\quad u\in \O,
\end{equation}
where $\O \subset \Y$ is an open neighborhood of   $\ou$,  $\vartheta:\Z\to \oR$ is proper, lsc, and sublinear,  and $\Xi:\O\to \Z$ is ${\cal C}^2$-smooth with   $\oz:=\Xi(\ou) = 0$, and where $\Z$ is a finite dimensional Hilbert space. 
Note that by  \cite[Definition~3.18 and Exercise~3.19]{rw}, we always have    $g(\oz) = 0$, since $g$ is proper and sublinear. As shown in \cite[Example~2.4]{sh03}, the class of ${\cal C}^2$-{decomposable} functions
is a generalization of ${\cal C}^2$-{\em cone reducible} sets in the sense of \cite[Definition~3.135]{bs} through which many important constrained optimization problems can be studied. 
Recall that a closed convex set $C\subset \Y$ is ${\cal C}^2$-{\em cone reducible} at $\ou\in \Y$ to the closed convex cone $\Th\subset \Z$ if there exist a neighborhood $\O\subset \Y$ of $\ou$ and a $\C^2$-smooth mapping $\Xi: \Y \to\Z$ such that 
\begin{equation}\label{codef}
C\cap\O =\big\{ u\in \O\, \big|\, \Xi(u) \in \Theta\big\}, \quad \Xi(\ou) = 0, \; \textrm{ and }\quad \nabla \Xi(\ou): \Y \to \Z\; \textrm{ is surjective}.
\end{equation}
Note that the   surjectivity condition in \eqref{codef}  is not  covered by \eqref{decom} and has its own counterpart, called the {\em nondegeneracy} condition; see \eqref{ndc}. 
Besides the indicator functions of ${\cal C}^2$-{\em cone reducible} sets, it was shown in \cite[Examples~2.1 and 2.3]{sh03} that polyhedral functions and the sum of the largest eigenvalues   of a symmetric matrix are 
${\cal C}^2$-{decomposable}. The later was extended for singular values of a matrix in  \cite[Example~5.3.18]{mi}. The readers can find more examples of ${\cal C}^2$-{decomposable} functions in \cite[Section~5.3.3]{mi}.
Note that CPLQ functions may not enjoy the composite representation in \eqref{decom} in general; see Remark~\ref{notdec}. 

Suppose that the proper  function $g:\Y\to \oR$ is ${\cal C}^2$-decomposable at $\ou\in \Y$ with representation \eqref{decom}.  Define the  (Lagrange) multiplier mapping $M_{\ou,g}:\Y\times\Y\tto\Z$   by
\begin{equation}\label{mpg}
M_{\ou,g}(y, w) = \big\{\mu \in \Z\, \big|\, \nabla \Xi(\ou)^*\mu = y, \, \mu\in \partial \vartheta(\Xi(\ou)+w)\big\}.
\end{equation}
Given $\oy\in \partial g(\ou)$, the set  $M_{\ou,g}(\oy,0)$ is called    the set of Lagrange multipliers associated with $(\ou,\oy)$. 
Below, we summerize some of the main properties of the composite representation in \eqref{decom}, used often in our proofs in this paper.
\begin{Proposition}\label{sulin}
Assume that $g:\Y\to \oR$ is ${\cal C}^2$-decomposable at $\ou\in \Y$ with representation \eqref{decom} and that $\oy\in \partial g(\ou)$ and $\omu\in M_{\ou,g}(\oy,0)$. 
Then the following properties hold.
\begin{enumerate}[noitemsep,topsep=0pt]
\item If the basic constraint qualification (BCQ) condition
\begin{equation}\label{bcq}
N_{\dom \vt} (\Xi(\ou))\cap \ker \nabla \Xi(\ou)^*=\{0\}
\end{equation}
holds, then $g$ is subdifferentially regular at $\ou$.
\item For any $z\in \dom \vt$, it holds that $\sub \vt(z)\subset \sub \vt(\Xi(\ou))$. 
\item It always holds that $K_{\vartheta}(\Xi(\ou), \omu)=N_{\sub \vt(\Xi(\ou))}(\omu)$. If, in addition,  \eqref{bcq} is satisfied, then $K_g(\ou,\oy)=N_{\sub g(\ou)}(\oy)$.
\item If the dual condition 
\begin{equation}\label{duq}
D (\sub \vartheta) (\Xi(\ou),\omu)(0)\cap\ker\nabla\Xi(\ou)^*=\{0\}
\end{equation}
holds, then
$K_g(\ou,\oy)^*= \nabla\Xi(\ox)^*K_{\vartheta}(\Xi(\ou), \omu)^*$.
\item If  \eqref{bcq} is satisfied, the subgradient mapping $\sub g$ is calm at $\ou$ for $\oy$.
\end{enumerate}
\end{Proposition}

\begin{proof} The conclusion in (a) is well-known; see \cite[Exercise~10.25]{rw}. The inclusion in (b) results from \cite[Corollary~8.25]{rw}.
Both claims in (c) result from (a) and \eqref{gr2}.
To prove (d),  the inclusion in \eqref{inc} tells us that the dual condition in \eqref{duq} yields \eqref{bcq}. Thus, it follows from the chain rule for the subderivative in \cite[Theorem~10.6]{rw} that 
\begin{align}
K_g(\ou,\oy) &= \big\{w\, \big|\, \d g(\ou)(w) =  \d\vt(\Xi(\ou))(\nabla\Xi(\ou)w) = \la \oy, w\ra= \la \omu,  \nabla\Xi(\ox)w\ra\big\} \nonumber\\ 
&= \big\{w\, \big|\, \nabla\Xi(\ox)w \in K_{\vartheta}(\Xi(\ou), \omu)\big\}.\label{critc}
\end{align}
Employing the chain rule for normal cones in \cite[Exercise~10.26]{rw} allows us to conclude that 
$$
K_g(\ou,\oy)^*=N_{K_g(\ou,\oy)}(0)=\nabla\Xi(\ou)^*N_{K_{\vartheta}(\Xi(\ou), \omu)}(0)=\nabla\Xi(\ou)^*K_{\vartheta}(\Xi(\ou), \omu)^*,
$$
which proves (d).  The claim in (e) results from (b) and \cite[Theorem~4.1]{r22}.
\end{proof}

As pointed out earlier, the equality in \eqref{gr1} plays a crucial role in our main result of this section. Below, we establish a chain rule for this property. 

\begin{Proposition}\label{ex:5}
Assume that $g:\Y\to \oR$ is ${\cal C}^2$-decomposable at $\ou\in \Y$ with representation \eqref{decom} and that $\oy\in \partial g(\ou)$ and $\omu\in M_{\ou,g}(\oy,0)$. 
If the outer function $\vartheta$ in \eqref{decom} enjoys the property 
\begin{equation}\label{ex5.1}
\Hat D(\partial \vartheta)(\Xi(\ou), \omu)(\xi) =\cl\big (D(\partial \vartheta)(\Xi(\ou), \omu)(\xi) - K_{\vartheta}(\Xi(\ou), \omu)^*\big) \quad\textrm{ for all }\quad \xi\in K_{\vartheta}(\Xi(\ou), \omu),
\end{equation}
and the nondegeneracy condition 
\begin{equation}\label{ndc}
\para\{\sub \vartheta(\Xi(\ou))\}\cap \ker \nabla \Xi(\ou)^*=\{0\}
\end{equation}
holds, then $g$ satisfies the same property, meaning that 
\begin{equation}\label{ex5.2}
\Hat D(\partial g)(\ou, \oy)(w) =\cl\big(D(\partial g)(\ou, \oy)(w) - K_g(\ou, \oy)^*\big)\quad\textrm{ for all }\quad w\in \Y.
\end{equation}
\end{Proposition}
\begin{proof}
Pick $w\in \dom \Hat D(\partial g)(\ou, \oy)$ and $\eta\in\Hat D(\partial g)(\ou, \oy)(w)$.  By definition, we find sequences $t_k\searrow0$, $w^k\to w$, $y^k \in \partial g(\ou+t_kw^k)$, 
and $\Hat y^k\to \oy$ with $\Hat y^k\in \partial g(\ou)$ such that $(y^k - \Hat y^k)/t_k \to \eta$. By the subdifferential chain rule from \cite[Example~10.8]{rw} available due to the assumed nondegeneracy condition, 
we find  $\mu^k \in \partial \vartheta(\Xi(\ou +t_kw^k))$ and $\Hat\mu^k\in\partial \vartheta(\Xi(\ou))$   such that $\nabla \Xi(\ou +t_kw^k)^*\mu^k = y^k$ and $\nabla \Xi(\ou)^*\Hat\mu^k = \Hat y^k$, respectively. 
Thus, we have $\mu^k\in M_{\ou,g}(y^k, t_k\nabla \Xi(\ou)w^k+o(t_k))$ and $\Hat\mu^k\in M_{\ou,g}(\Hat y^k, 0)$. 
Note that it follows from \eqref{ssub} and Proposition~\ref{sulin}(c) that 
\begin{equation}\label{nddu}
D(\partial \vartheta)(\Xi(\ou), \omu)(0)= K_{\vartheta}(\Xi(\ou), \omu)^*=T_{\sub \vt(\Xi(\ou))}(\omu)\subset \para\{\sub \vartheta(\Xi(\ou))\},
\end{equation}
where the last inclusion results from the definition of the tangent cone. Thus, we conclude from the assumed nondegeneracy condition that 
  the dual condition in \eqref{duq} is satisfied. Appealing now to Theorem~\ref{calag}(d), we can conclude that both sequences $\{\mu^k\b$ and $\{\Hat \mu^k\b$ converge to $\omu$ as $k\to\infty$. 
By passing to a subsequence, we can assume without loss of generality that  the sequence $\{\nabla\Xi(\ox)^*(\mu^k-\Hat\mu^k)/t_k\b$ is convergent, since 
\begin{align}
\frac{y^k-\Hat y^k}{t_k} &=\frac{ \nabla \Xi(\ou +t_kw^k)^*\mu^k - \nabla \Xi(\ou)^*\Hat\mu^k}{t_k}\nonumber\\
&=\Big(\frac{\nabla \Xi(\ou +t_kw^k) -\Xi(\ou)}{t_k}\Big)^*\mu^k +\nabla\Xi(\ou)^*\frac{\mu^k-\Hat\mu^k}{t_k}.\label{ex5.3}
\end{align}
We claim now that the sequence $\{\|\mu^k-\Hat\mu^k\|/t_k\b$ is bounded. Otherwise, passing to a subsequence, if necessary, we can assume
 that $\|\mu^k-\Hat\mu^k\|/t_k\to\infty$, which implies that $\mu^k -\Hat\mu^k \neq 0$ for any  $k$ sufficiently large. 
Again, we can assume by passing to a subsequence, if necessary,  that $\{(\mu^k-\Hat\mu^k)/\|\mu^k-\Hat\mu^k\|\b$ converges to some $\xi \neq 0$. 
Thus, we get $\xi\in \ker \nabla\Xi(\ox)^*$, since  
\begin{equation*}
\nabla\Xi(\ox)^*\xi = \lim_{k\to\infty} \nabla\Xi(\ox)^*\frac{\mu^k-\Hat\mu^k}{\|\mu^k-\Hat\mu^k\|} = \lim_{k\to\infty} \nabla\Xi(\ox)^*\frac{\mu^k-\Hat\mu^k}{t_k}\cdot\frac{t_k}{\|\mu^k-\Hat\mu^k\|} =0. 
\end{equation*}
On the other hand, it follows from Proposition~\ref{sulin}(b) that $\mu^k\in \partial \vt(\Xi(\ou+t_kw^k))\subset \partial \vt(\Xi(\ou))$, which in turn yields 
 $(\mu^k-\Hat\mu^k)/\|\mu^k-\Hat\mu^k\| \in \para \{\partial \vt(\Xi(\ou))\}$.  Thus, we arrive at  $\xi \in \para\{ \partial \vt(\Xi(\ou)\}$, which leads us to 
 $\xi =0$ due to \eqref{ndc},  a contradiction.  This  tells us that  $\{(\mu^k-\Hat\mu^k)/t_k\b$ is convergent to some $\zeta\in \Z$. Recall that $\mu^k \in \partial \vartheta(\Xi(\ou +t_kw^k))$ and $\Hat\mu^k\in\partial \vartheta(\Xi(\ou))$.
Since $\big(\Xi(\ou +t_kw^k)-\Xi(\ou)\big)/t_k\to \nabla\Xi(\ou)w$, it follows from Definition~\ref{ssgd} that 
 $\zeta \in \Hat D(\partial \vt)(\Xi(\ou), \omu)(\nabla \Xi(\ou)w)$. 
Passing to the limit in  \eqref{ex5.3}  brings us to 
\begin{equation*}
\eta - \nabla^2\la \omu, \Xi\ra(\ou)w =\nabla\Xi(\ou)^*\zeta \in \nabla \Xi(\ou)^*\Hat D(\partial \vt)(\Xi(\ou), \omu)(\nabla \Xi(\ox)w),
\end{equation*}
which in turn implies via \eqref{ex5.1}, Theorem~\ref{dess}(b), and Proposition~\ref{sulin} that 
\begin{align*}
\eta &\in \nabla^2\la \omu, \Xi\ra(\ou)w + \nabla\Xi(\ou)^*\cl\big( \Hat D(\partial \vt)(\Xi(\ou), \omu)(\nabla \Xi(\ou)w) - K_{\vartheta}(\Xi(\ou), \omu)^*\big)\nonumber\\
&=  \cl\big(  \nabla^2\la \omu, \Xi\ra(\ou)w +\nabla\Xi(\ou)^*\Hat D(\partial \vt)(\Xi(\ou), \omu)(\nabla \Xi(\ou)w) - \nabla\Xi(\ou)^*K_{\vartheta}(\Xi(\ou), \omu)^*\big)\\
&= \cl\big(D(\sub g)(\ou,\oy)(w)- K_g(\ou,\oy)^*\big).
\end{align*}
This proves the inclusion `$\subset$'  in \eqref{ex5.2}. To justify the opposite inclusion, observe from \eqref{inc} and \eqref{ndc} that the BCQ in \eqref{bcq} is 
satisfied. By Proposition~\ref{sulin}, $g$ is subdifferentially regular at $\ou$. Moreover, it follows from Theorem~\ref{dess}(b) that $\sub g$ is proto-differentiable at $\ou$ for $\oy$.
Thus, Proposition~\ref{prop:gr}, together with \eqref{gr2}, proves the inclusion `$\supset$' in \eqref{ex5.2} and hence completes the proof.  
\end{proof}

Using the established chain rule in the proposition above, we are going next to show that the inclusion in \eqref{gr1}  becomes equality for the second-order cone.

\begin{Example}\label{ex:6}
Let  ${\cal Q}$ stand for the second-order cone in $\R^n$ described by
\begin{equation*}
{\cal Q}:=\big\{u = (u_0, u_r) \in \R\times\R^{n-1}\, \big|\, \|u_r\|\leq u_0\big\}.
\end{equation*}
Our goal is to show that the inclusion in \eqref{gr1} becomes equality when $f=\dd_{\cal Q}$. 
To achieve this goal, we should consider there different cases for a point $\ou\in {\cal Q}$. If $\ou\in \inte {\cal Q}$, 
it is easy to see that ${\cal Q} $ can be locally represented by the set $\{u\in \O\, | \;\Xi(u)\in \Th\}$,  where $\Xi:\R^n \to \{0\}$ and $\Th=\{0\}$. 
If  $\ou\in \bd {\cal Q}\setminus \{0\}$, one can see that ${\cal Q} $ can be locally represented by $\{u\in \O\, | \; \Xi(u)\in \Th\}$,  where $\Xi:\R^n \to \R$ is defined by $\Xi(u)=\|u_r\|-u_0$ for any 
$u=(u_0,u_r)\in \R\times \R^{n-1}$ and $\Th=\R_-$. 
Thus, in both cases, ${\cal Q}$ is  ${\cal C}^2$-cone reducible to a polyhedral convex cone in the sense of \eqref{codef} and hence $\dd_{\cal Q}$ is ${\cal C}^2$-decomposable
with the outer function $\vt$ in the composite representation in \eqref{decom} being either $\dd_{\{0\}}$ or $\dd_{\R_-}$. Since these functions are CPLQ, it follows from 
  Proposition~\ref{ex:5} and Example~\ref{ex:4} that the inclusion in \eqref{gr1} becomes equality for $f=\dd_{\cal Q}$ at any nonzero $\ou \in {\cal Q}$ and $\oy \in N_{\cal Q}(\ou)$.
  
  The last case to consider is $\ou=0$. Take $\oy\in N_{\cal Q}(\ou)$ and observe from   \cite[Corollary~3.4]{HMS20} that 
  $N_{\cal Q}$ is proto-differentiable at $\ou$ for $\oy$ and that $DN_{\cal Q}(\ou, \oy)(w) = N_{K_{\cal Q}(\ou, \oy)}(w)$ for any $w\in \R^n$ with $K_{\cal Q}(\ou, \oy) = {\cal Q}\cap [\oy]^\perp$. 
To achieve our goal,  we only need  to justify  the inclusion
\begin{equation*}\label{ex6.1}
\Hat DN_{\cal Q}(\ou, \oy)(w) \subset \cl\big(DN_{\cal Q}(\ou, \oy)(w)-K_{\cal Q}(\ou, \oy)^*\big)=\cl\big(N_{K_{\cal Q}(\ou, \oy)}(w)-K_{\cal Q}(\ou, \oy)^*\big) \quad \mbox{for all}\;\; w\in \R^n,
\end{equation*}
since the opposite inclusion always holds according to Proposition~\ref{prop:gr}.  By  \cite[Corolalary~11.25(b)]{rw}, we can equivalently express the latter  as
\begin{equation}\label{ex6.2}
\Hat DN_{\cal Q}(\ou, \oy)(w) \subset \big(T_{K_{\cal Q}(\ou, \oy)}(w)\cap-K_{\cal Q}(\ou, \oy)\big)^*.
\end{equation}
To prove this inclusion, pick $w\in \dom \Hat D N_{\cal Q}(\ou, \oy)$, $\eta \in \Hat D N_{\cal Q}(\ou, \oy)(w)$ and conclude from \eqref{gr3} that $w\in K_{\cal Q}(\ou, \oy)$. 
To verify \eqref{ex6.2}, we must   show that $\la \eta, q\ra \leq 0$ for any $q \in C(w)$, where $C(w):=T_{K_{\cal Q}(\ou, \oy)}(w)\cap-K_{\cal Q}(\ou, \oy)$.
If $w=0$,  we clearly have 
$$
C(w) = K_{\cal Q}(\ou, \oy)\cap-K_{\cal Q}(\ou, \oy)=({\cal Q}\cap [\oy]^\perp)\cap-({\cal Q}\cap [\oy]^\perp)=\{0\},
$$ 
which proves  \eqref{ex6.2}. If $w\neq 0$, it follows from $\eta \in \Hat D N_{\cal Q}(\ou, \oy)(w)$ that there exist
 sequences $t_k\searrow0$, $w^k\to w$, $y^k \in N_{\cal Q}(w^k)$, and $\Hat y^k\to \oy$ with $\Hat y^k\in -{\cal Q}$ such that $(y^k - \Hat y^k)/t_k \to \eta$.
We split our verification of \eqref{ex6.2} into three cases depending on the position of $\oy$ in ${\cal Q}$:

\begin{itemize}[noitemsep,topsep=0pt]
\item [(i)]   $\oy \in -\inte{\cal Q}$.  In this case, we have $K_{\cal Q}(\ou, \oy) = \{0\}$. 
Since $w\in K_{\cal Q}(\ou, \oy)$, we get $w=0$, which is not possible. 

\item [(ii)]   $\oy\in -(\bd{\cal Q})\setminus\{0\}$. In this case, since $0\neq w\in K_{\cal Q}(\ou, \oy)$, it is not hard to see that 
   $K_{\cal Q}(\ou, \oy) = \{tw|\; t\ge 0\}$.  This leads us to  
   $$
   C(w) =T_{K_{\cal Q}(\ou, \oy)}(w)\cap-K_{\cal Q}(\ou, \oy)= [w] \cap \{tw|\; t\le 0\}= \{tw|\; t\le 0\}.
   $$
On the other hand, it follows from  $y^k \in N_{\cal Q}(w^k)$ and $\Hat y^k\in -{\cal Q}$ that 
\begin{equation*}
\la \frac{ y^k- \Hat y^k}{t_k}, w^k\ra = -\la \frac{\Hat y^k}{t_k}, w^k\ra\geq 0,
\end{equation*}
which in turn results in $\la \eta, w\ra \geq 0$. Pick now  $q\in C(w)$. So, $q=tw$ for some $t\le 0$. Combining these shows $\la \eta, q\ra \leq 0$, which again proves \eqref{ex6.2}. 

\item [(iii)]  $\oy = 0$. In this case,  we have $K_{\cal Q}(\ou, \oy) = {\cal Q}$. Recall that $  w\in K_{\cal Q}(\ou, \oy)$. 
If $w\in \inte{\cal Q}$, it follows from $w^k\to w$ that  $w^k \in\inte{\cal Q}$ for any $k$ sufficiently large and therefore $y^k =0$ for any such $k$. 
Thus, $-\Hat y^k/t_k \to \eta$,  which, together with $\Hat y^k\in -{\cal Q}$, yields $\eta \in {\cal Q}$. 
Note that  for such $w$, we have 
$$
   C(w) =T_{K_{\cal Q}(\ou, \oy)}(w)\cap-K_{\cal Q}(\ou, \oy)= \R^n \cap  -{\cal Q}=-{\cal Q}.
   $$
Thus, we obtain $\eta \in   C(w)^*$, which again proves \eqref{ex6.2}. 
If $ w\in \bd{\cal Q}$,    we  conclude that 
$$
C(w) = T_{K_{\cal Q}(\ou, \oy)}(w)\cap-K_{\cal Q}(\ou, \oy)=T_{\cal Q}(w)\cap- {\cal Q} =\{tw|\; t\le 0\}.
$$ Using the same argument as case (ii) shows that  $\eta \in C(w)^*$ and hence proves \eqref{ex6.2}. 
\end{itemize}
Combining  these cases shows that the inclusion in \eqref{gr1} holds as equality for $f = \delta_{\cal Q}$ when $\ou=0$. 
\end{Example}

\begin{Remark}\label{notdec}
Note  that  CPLQ functions may not be ${\cal C}^2$-decomposable. 
To delineate this, it is important to mention that most functions,   enjoying the composite representation in \eqref{decom}, automatically satisfy the nondegenracy condition in \eqref{ndc}; see \cite[Examples~2.1 and 2.3]{sh03}. 
Assume that  a function $g:\Y\to \oR$ is ${\cal C}^2$-decomposable at $\ou$, that the nondegenracy condition in \eqref{ndc} holds, and that $\oy\in \sub g(\ou)$. 
Employing now Theorem~\ref{dess} and \eqref{ssvt} tells us that the second subderivaive of $g$ at $\ou$ for $\oy$ has a single quadratic piece, namely $ \la \omu, \nabla^2\Xi(\ou)(w, w)\ra$, whenever $\nabla \Xi(\ou)w\in K_{\vt}(\Xi(\ou),\oy)$.
If $g$ were a CPLQ, it would follow from \eqref{pwfor} that $\d^2 g(\ou,\oy)$ may have several different quadratic pieces depending on how many polyhedral convex sets $C_i$ exist in the representation of $g$; see the beginning of Example~\ref{ex:4}
for the definition of $C_i$. That contradicts our earlier observation about $\d^2 g(\ou,\oy)$ and demonstrates that CPLQ functions cannot be ${\cal C}^2$-decomposable in general. 
\end{Remark}

We are now going to establish an error bound for the  KKT system in \eqref{vs} associated with the composite optimization problem \eqref{comp}.  Such a result was perviously established under the 
second-order sufficient condition (SOSC) for different classes of constrained optimization problems in \cite[Theorem~5.9]{ms18}. For a solution $(\ox,\oy)$ to the KKT system \eqref{vs}, the latter result motivates us to consider 
the SOSC
\begin{equation}\label{sosc1}
\begin{cases}
\langle\nabla_{xx}^2L(\bar x,\oy)w,w\rangle+\d^2g (\Phi(\bar x),\oy) (\nabla \Phi(\ox)w ) >0 \quad \mbox{ for all }\; w\in\D\setminus \{0\},\\
\textrm{where }\;\D:= K_\Th(\ox,-\nabla_xL(\ox, \oy))\cap\big\{w\in \X\, \big|\, \nabla\Phi(\ox)w\in K_g(\Phi(\ox),\oy)\big\},
\end{cases}
\end{equation}
and investigate whether our desired error bound can be derived under this condition. Note that the SOSC in \eqref{sosc1} boils down to the classical second-order sufficient condition 
formulated for different classes of constrained optimization problems including NLPs, second-order cone programming, and SDPs; see \cite[Section~6]{mms}
for a detailed discussion about the SOSC. For the composite programming problems  in \eqref{comp} with $g$ therein being CPLQ, one can find recent developments about the SOSC in \cite{HaS21, s20}.

It is well-known that the SOSC alone does not ensure the error bound for the KKT system in \eqref{vs}; see the discussion after \cite[Theorem~5.9]{ms18}. 
What is needed further is calmness of a multiplier mapping associated with  the composite problem in \eqref{comp}, which is defined as the one for $g$ in \eqref{mpg}.
Given  $\ox\in \X$ and $\psi $ taken from \eqref{comp},  the {\em multiplier mapping} $M_{\ox,\psi}\colon\X\times\Y\rightrightarrows\Y$, associated with the canonically perturbed KKT system \eqref{vs}, is defined by
\begin{equation}\label{mulmap}
M_{\ox, \psi}(v,w):=\big\{y\in\Y \;\big|\;v\in\nabla_x L(\ox,y)+N_\Th(\ox),\;y\in \partial g\big(\Phi(\ox)+w\big)\big\},\;\;(v,w)\in\X \times\Y.
\end{equation}
It is easy to see that if $\ox$ is a local minimum of the composite problem in \eqref{comp},  $M_{\ox,\psi}(0,0)$ reduces to the set of Lagrange multipliers
associated with $\ox$. Note also that for any $\oy\in M_{\ox,\psi}(0,0)$, the pair $(\ox,\oy)$ is a solution to the KKT system in \eqref{vs}.  While, throughout this paper,
we always assume that a Lagrange multiplier exists, it is well-known in general  that the existence of such a Lagrange multiplier requires a {\em constraint qualification}. 
Getting into the question of which constraint qualification to choose is unnecessary for our developments in this paper, however. 

Given a solution $(\ox,\oy)$ to the KKT system \eqref{vs}, the calmness of the multiplier mapping $M_{\ox, \psi}$ in \eqref{mulmap}  at $(0,0)$ for $\oy$ amounts to the existence of  a positive constant $\tau$ 
and neighborhoods $U$ of $(0,0)$ and $V$ of $\oy$ such that
\begin{equation*}
M_{\ox, \psi}(v,w)\cap V\subset M_{\ox, \psi}(0,0)+\tau\Vert (w,v)\Vert \B\quad \mbox{for all}\;\;(v,w)\in U.
\end{equation*}
It follows from \cite[Theorem~3H.3]{dr} that this property is equivalent to the metric subregualrity of its inverse mapping $M_{\ox, \psi}^{-1}$  at $\oy$ for $(0,0)$, which means that 
there exist  a positive constant $\ell$ and a  neighborhood   $V$ of $\oy$  for which the estimate
\begin{equation}\label{calmm}
\dist\big(y,M_{\ox, \psi}(0,0)\big)\le\ell\,\big(\dist(-\nabla_x L(\ox,y), N_\Th(\ox))+\dist(\Phi(\ox),(\partial g)^{-1}(y))\big)\quad \mbox{for all}\; y\in V
\end{equation}
holds. Below, we record some conditions under which the clamness of the multiplier mapping $M_{\ox, \psi}$ is satisfied.

\begin{Proposition} \label{cmpr} Assume that $(\ox,\oy)$ is a solution to the KKT system \eqref{vs}. Then 
the multiplier mapping $M_{\ox, \psi}$ is calm at $(0,0)$ for $\oy$ if one of the following conditions holds:
\begin{itemize}[noitemsep,topsep=0pt]
\item [\rm (a)] the convex function $g$ in \eqref{comp} is CPLQ;
\item [\rm (b)] the subgradient  mapping $\sub g$ is calm at $\Phi(\ox)$ for $\oy$ and there exists $\tilde y\in M_{\ox, \psi}(0,0)$ such that 
$-\nabla_x L(\ox,\tilde y) \in  \ri N_\Th(\ox)$ and $\tilde y\in \ri \sub g(\Phi(\ox))$.
\end{itemize}
\end{Proposition}
\begin{proof} Observe first that the  set $M_{\ox, \psi}(0,0)$  can be equivalently written as 
 $$
M_{\ox, \psi}(0,0)=\Omega\cap \sub g(\Phi(\ox))\quad \mbox{with}\quad \Omega:=\big\{y\in\Y\;\big|\; 0\in \nabla_x L(\ox,y)+N_\Th(\ox)\big\}.
 $$
 If  (a) holds,   $\sub g(\Phi(\ox))$ is a polyhedral convex set. By assumption, $\Omega$ is a polyhedral convex set as well.  It follows then from \cite[Theorem~8.35]{io}
that there is a constant $\rho\ge 0$ such that  for any $y\in Y$ the estimate 
$$
\dist(y, M_{\ox, \psi}(0,0))\le \rho\big(\dist(y,\Omega)+ \dist(y, \sub g(\Phi(\ox))\big)
$$
holds. If (b) holds, the same estimate holds for any $y$ sufficiently close to $\oy$ due to the existence of $\tilde y$, which can be equivalently translated as $\ri \Omega\cap \ri  \sub g(\Phi(\ox))\neq \emptyset$; see \cite[Corollary 3]{bbl}. 
Using the classical Hoffman lemma (cf. \cite[Lemma~3C.4]{dr}) gives us a constant $\rho'\ge 0$ such that the estimate 
$$
\dist(y,\Omega) \le \rho' \dist\big(-\nabla_xL(\ox,y),N_\Th(\ox)\big)
$$
is satisfied for any $y\in \Y$. If $g$ is CPLQ, it follows from the proof of \cite[Theorem~11.14(b)]{rw} that $\gph\partial g$ is a union of finitely many  polyhedral convex  sets. 
So, we conclude from \cite[Proposition~1]{rob1} (see also \cite[Theorem~3D.1]{dr}) that the subgradient  mapping $\sub g$ is calm at $\Phi(\ox)$ for $\oy$. 
Thus, if either (a) or (b) holds, $(\sub g)^{-1}$ is metrically subregular at $\oy$ for $\Phi(\ox)$ (cf. \cite[Theorem~3H.3]{dr}), meaning that  there are a positive constant $\rho''$ and a  neighborhood   $V$ of $\oy$  for which the estimate
 $$
  \dist(y, \sub g(\Phi(\ox))\big) \le \rho''\, \dist(\Phi(\ox), (\sub g)^{-1}(y)\big)
 $$
holds for any $y\in V$. Combining all these estimates proves the calmness of the multiplier mapping $M_{\ox, \psi}$   at $(0,0)$ for $\oy$. 
\end{proof}

It is worth mentioning that the existence of $\tilde y$  in Proposition~\ref{cmpr}(b) boils down to the classical {\em strict complementarity} assumption in case $\Th=\X$ in the composite problem in \eqref{comp}, 
which reads as the existence of a multiplier $\tilde y$ such that $\nabla_x L(\ox,\tilde y) =0$ and $\tilde y\in \ri \sub g(\Phi(\ox))$. Furthermore, 
according to \cite[Theorem~3.3]{ag}, the calmness of the   subgradient  mapping $\sub g$ in Proposition~\ref{cmpr}(b) is equivalent to the conjugate function $g^*$
satisfying the quadratic growth condition 
$$
g^*(y)\ge g^*(\oy)+\la \Phi(\ox), y-\oy\ra+\kappa\,\dist\big(y,\sub g( \Phi(\ox))\big)\quad \mbox{whenever}\;\;y\in V,
$$
where $\kappa$ is a positive constant and $V$ is a neighborhood of $\oy$. It is known that this calmness condition holds for the subgradient mapping of CPLQ functions (cf. \cite[Proposition~1]{rob1}) as well as the   ${\cal C}^2$-decomposable  functions
under the condition \eqref{bcq} (cf. Propsoition~\ref{sulin}(e)).

To proceed, define  the {\em residual}  function   $r\colon\X\times\Y\to\R$ of the KKT system \eqref{vs} by
\begin{equation}\label{res}
r(x,y):=\dist(-\nabla_x L(x,y), N_\Th(x))+\Vert\Phi(x)-{\rm prox}_g\big(\Phi(x)+y\big)\Vert,\;\;(x,y)\in\X \times\Y. 
\end{equation}
It is easy to see that a pair $(x,y)$ is a solution to the KKT system in \eqref{vs} if and only if $r(x,y)=0$.
The residual function $r$, indeed, measures the violation of the KKT system in \eqref{vs} for a given pair $(x,y)$ and plays an important role as a surrogate  for  
the tolerance parameter $\ek$ in \eqref{xkk}.  

\begin{Theorem}\label{error} 
Assume that  $(\ox,\oy)$ is a solution to the KKT system \eqref{vs} for which the SOSC in \eqref{sosc1}
holds. Assume further that   one of the following conditions holds:
\begin{itemize}[noitemsep,topsep=0pt]
\item [\rm (a)]  the convex function $g$ in \eqref{comp} is ${\cal C}^2$-decomposable and the nondegeneracy condition in \eqref{ndc} is satisfied for $\ou=\Phi(\ox)$;
\item [\rm (b)] the convex function $g$ in \eqref{comp}  is twice epi-differentiable at $\Phi(\ox)$ for $\oy$ and the equality in \eqref{ex5.2} holds for $\ou=\Phi(\ox)$ and $\oy$.
\end{itemize}
If the multiplier mapping $M_{\ox,\psi}$ in \eqref{mulmap} is calm at $(0,0)\in \X\times \Y$ for $\oy$, then there exist constants $\gamma>0$ and $\kappa\ge 0$ such that
\begin{equation}\label{eb}
\Vert x-\ox\Vert+\dist\big(y,M_{\ox,\psi}(0,0)\big)\le\kappa\,r(x,y)\quad\mbox{ for all }\;\;(x,y)\in\B_{\gamma}(\ox,\oy).
\end{equation}
\end{Theorem} 
\begin{proof}
We begin by showing that  that there exists a neighborhood $U$ of $(\ox,\oy)$ such that 
\begin{equation}\label{ebb}
\Vert x-\ox\Vert=O\big(r(x,y)\big)\quad\textrm{for all}\quad(x,y)\in U.
\end{equation}
If this claim fails, we find a sequence $(\xk,y^k)\to(\ox,\oy)$  such that 
\begin{equation*}
t_k:= \Vert\xk-\ox\Vert>k\,r(\xk,y^k)\ge 0\quad\mbox{for all}\;\,k\in \N,
\end{equation*}
meaning that $r(\xk,y^k)=o(t_k )$. Thus,  by \eqref{res}, we find $\alk\in \X$ and $\bek\in \Y$ such that $ \alk = o(t_k )$ and  $\bek=o(t_k )$, where 
\begin{equation}\label{eb3}
-\nabla_x L(\xk,y^k)+\alk\in N_\Th(\xk)\quad\textrm{and }\quad\bek=\Phi\big(\xk)-{\rm prox}_g (\Phi(\xk)+y^k\big).
\end{equation}
It then follows from the first identity in \eqref{prm} that 
\begin{equation}\label{eb2}
y^k +\bek \in \partial g(\Phi(\xk)-\bek).
\end{equation}
 The calmness of multiplier mapping $M_{\ox,\psi}$  at $(0,0)$ for $\oy$ implies that there exist 
  a positive constant $\ell$ and a  neighborhood   $V$ of $\oy$  for which the estimate in \eqref{calmm}
holds. This allows us to conclude for any $k$ sufficiently large that 
\begin{equation*}
\dist\big(y^k+\bek,M_{\ox,\psi}(0,0)\big)\le\ell\big(\dist(-\nabla_x L(\ox,y^k+\bek), N_\Th(\ox))+\dist(\Phi(\ox),(\partial g)^{-1}(y^k+\bek))\big).
\end{equation*}
  Recall that $\Th$ is a polyhedral set and that $\xk \to \ox$ as $k\to \infty$. We then deduce that $N_\Th(\xk) \subset N_\Th(\ox)$ for $k$ sufficiently large and get the estimates
\begin{align*}
\dist(-\nabla_x L(\ox,y^k+\bek), N_\Th(\ox)) \le&\, \dist(-\nabla_x L(\ox,y^k+\bek), N_\Th(\xk))\nonumber\\
\leq&\, \|-\nabla_x L(\ox,y^k+\bek)- (-\nabla_xL(\xk, y^k)+\alk)\|\nonumber\\
\leq&\,  \Vert\nabla \varphi(\xk)-\nabla \varphi(\ox)\Vert+\Vert(\nabla\Phi(\xk)-\nabla\Phi(\ox))^*y^k\Vert \\
&+\Vert \nabla\Phi(\ox)^*\bek\Vert+\|\alk\|=O(t_k ),
\end{align*}
where the last equality comes from the Lipschitz continuity of   mappings $\nabla \varphi$ and $\nabla\Phi$ around $\ox$ and the estimates in \eqref{eb3}.
Moreover, by \eqref{eb2}, we have   $\Phi(\xk)-\bek\in (\partial g)^{-1}(\lmk+\bek)$, which coupled with the Lipschitz continuity of $\Phi$ around $\ox$ and  $\bek=o(t_k )$ leads us to 
\begin{equation*}
\dist\big(\Phi(\ox),(\partial g )^{-1}(\lmk+\bek)\big)\le\Vert\Phi(\xk)-\bek-\Phi(\ox)\Vert=O(t_k )
\end{equation*}
 for any $k$ sufficiently large. Combining these estimates demonstrates that 
 $$
 \dist\big(y^k+\bek,M_{\ox,\psi}(0,0)\big)=O(t_k)
 $$
 for any $k$ sufficiently large.  Since $M_{\ox,\psi}(0,0)$ is a closed convex set, set $\Hat y^k:=P_{M_{\ox,\psi}(0,0)}(y^k)$. This, the above estimate, and $\bek=o(t_k )$ allow us  to 
 obtain $y^k-\Hat y^k=O(t_k )$. Passing to subsequences if necessary, we can assume that   
\begin{equation}\label{eb4}
\dfrac{\xk-\ox}{t_k }\to w\ne 0\quad\textrm{and}\quad\dfrac{y^k-\Hat y^k}{t_k }\to\eta\;\mbox{ as }\;k\to\infty.
\end{equation}
Recall from \eqref{vs} that $- \nabla_xL(\ox, \oy)\in N_\Th(\ox)$. This, coupled with the first inclusion in  \eqref{eb3} and the 
reduction lemma for polyhedral convex sets (cf. \cite[Lemma~2E.4]{dr}), allows us to conclude for any $k$ sufficiently large that
\begin{equation*}
 - \big(\nabla_xL(\xk, y^k) - \nabla_xL(\ox, \oy)\big) + \alk \in N_{K_\Th(\ox, -\nabla_xL(\ox, \oy))}(\xk - \ox).
\end{equation*}
By the definition of $\Hat y^k$, we have $- \nabla_xL(\ox, \Hat y^k) \in N_\Th(\ox)$. Applying again the reduction lemma from \cite[Lemma~2E.4]{dr} brings us to 
\begin{equation*}
 - \big(\nabla_xL(\ox, \Hat y^k) - \nabla_xL(\ox, \oy)\big) \in N_{K_\Th(\ox, -\nabla_xL(\ox, \oy))}(0) = K_\Th(\ox, -\nabla_xL(\ox, \oy))^*
\end{equation*}
for any $k$ sufficiently large. On the other hand, we have
\begin{align*}
\nabla_xL(\xk, y^k)- \nabla_xL(\ox, \oy) &=\nabla_xL(\xk, \oy) - \nabla_xL(\ox,\oy) + \nabla\Phi(x^k)^*(y^k-\oy)\\
 &=\nabla^2_{xx}L(\ox,\oy)(\xk-\ox)+\nabla\Phi(\xk)^*(y^k-\Hat y^k)+ \nabla\Phi(\ox)^*(\Hat y^k - \oy)+o(t_k),
\end{align*}
and 
$$
\nabla_xL(\ox, \Hat y^k) - \nabla_xL(\ox, \oy)=  \nabla\Phi(\ox)^*(\Hat y^k - \oy).
$$
Combining these and remembering that $\alk=o(t_k )$ lead us  to the inclusion
\begin{align*}
\nabla^2_{xx}L(\ox,\oy)(\frac{\xk-\ox}{t_k})&+\nabla\Phi(\xk)^*(\frac{y^k-\Hat y^k}{t_k})+\frac{o(t_k)}{t_k}\\
&\in K_\Th(\ox, -\nabla_xL(\ox, \olm))^*-N_{K_\Th(\ox, -\nabla_xL(\ox, \oy))} (\frac{\xk - \ox}{t_k} )\\
&\subset K_\Th(\ox, -\nabla_xL(\ox, \olm))^* -N_{K_\Th(\ox, -\nabla_xL(\ox, \oy))}(w)
\end{align*}
for any $k$ sufficiently large, where the last inclusion results from the first limit in \eqref{eb4} and  $K_\Th(\ox, -\nabla_xL(\ox, \olm))$ being a polyhedral convex set. 
The right-hand side of this inclusion is the sum of two polyhedral convex sets, which is again a polyhedral convex set and hence  is  closed. Passing to the limit  brings us to 
\begin{equation*}\label{eb9}
\nabla^2_{xx}L(\ox,\oy)w+\nabla\Phi(\ox)^*\eta \in K_\Th(\ox, -\nabla_xL(\ox, \oy))^* - N_{K_\Th(\ox, -\nabla_xL(\ox, \oy))}(w).
\end{equation*}
This implies that $w\in K_\Th(\ox, -\nabla_xL(\ox, \oy))$ and  that 
there are $\nu_1\in K_\Th(\ox, -\nabla_xL(\ox, \oy))^*$ and $\nu_2 \in N_{K_\Th(\ox, -\nabla_xL(\ox, \oy))}(w) = K_\Th(\ox, -\nabla_xL(\ox, \oy))^*\cap[w]^\bot$ such that 
$\nabla^2_{xx}L(\ox,\oy)w+\nabla\Phi(\ox)^*\eta = \nu_1 - \nu_2$, which yields 
\begin{equation}\label{eb14.5}
\la\nabla^2_{xx}L(\ox,\oy)w, w\ra+\la \eta, \nabla\Phi(\ox)w\ra = \la\nu_1, w\ra - \la \nu_2, w\ra \leq 0.
\end{equation}
We claim now $\nabla\Phi(\ox)w \in K_g(\Phi(\ox),\oy)$. 
It follows from the convexity of  $g$, $\oy \in \partial g(\Phi(\ox))$ and \cite[Theorem~8.30]{rw}  that  $\la\oy,\nabla\Phi(\ox)w\ra\leq \d g(\Phi(\ox))(\nabla\Phi(\ox)w)$. 
To prove the opposite inequality, we deduce from \eqref{eb2} that
\begin{align*}
\la y^k+\bek,\Phi(\xk) -\Phi(\ox)-\bek\ra&\geq g(\Phi(\xk) - \bek) - g(\Phi(\ox))\nonumber\\
 &= g\big(\Phi(\ox) + t_k\frac{\Phi(\xk) -\Phi(\ox)-\bek}{t_k}\big)-g(\Phi(\ox)).
\end{align*}
Dividing both sides by $t_k $ and  passing then to the limit result in   $\la\oy,\nabla\Phi(\ox)w\ra\ge \d g(\Phi(\ox))(\nabla\Phi(\ox)w)$. Thus we get   $\la\oy,\nabla\Phi(\ox)w\ra= \d g(\Phi(\ox))(\nabla\Phi(\ox)w)$, which proves our claim.
To arrive at a contradiction with the SOSC in \eqref{sosc1}, we need to   prove that the second subderivative of $g$ satisfies   
\begin{equation}\label{eb10}
\langle\eta,\nabla\Phi\big(\ox)w\rangle\ge\d^2 g(\Phi(\ox),\oy)(\nabla\Phi(\ox)w).
\end{equation}
Assume first that    (b) holds, meaning that $g$ satisfies in \eqref{ex5.2}. By \eqref{eb4} and $\bek=o(t_k )$, we have 
$$
\big(\frac{\Phi(\xk)-\bek-\Phi(\ox)}{t_k}, \frac{y^k +\bek- \Hat y^k}{t_k}\big)\to (\nabla \Phi(\ox)w, \eta).
$$
This, together with \eqref{eb2}, the fact $\Hat y^k\in  \partial g(\Phi(\ox))$ and Definition~\ref{ssgd}, tells us that $\eta\in \Hat D(\partial g)(\Phi(\ox), \oy)(\nabla \Phi(\ox)w)$.
By \eqref{ex5.2}, we find  sequences $\{\eta_1^k\b \subset D(\partial g)(\Phi(\ox), \oy)(\nabla \Phi(\ox)w)$ and $\{\eta_2^k\b \subset K_g(\Phi(\ox), \oy)^*$ for which we have  
$\eta_1^k-\eta_2^k\to \eta$ as $k\to \infty$. It follows from \cite[Theorem~13.40]{rw} and \cite[Lemma~3.6]{chnt} that $\d^2 g(\Phi(\ox),\oy)(\nabla\Phi(\ox)w)=\la \eta_1^k, \nabla\Phi(\ox)w\ra$ and from $\nabla\Phi(\ox)w \in K_g(\Phi(\ox),\oy)$ that 
$\la \eta_2^k, \nabla\Phi(\ox)w\ra \le 0$ for all $k\in \N$. So, we get 
\begin{equation*}
\la \eta_1^k-\eta_2^k, \nabla\Phi(\ox)w\ra \ge \d^2 g(\Phi(\ox),\oy)(\nabla\Phi(\ox)w). 
\end{equation*}
Passing to the limit then proves \eqref{eb10} when (b) holds. Assume now that (a) is satisfied. 
It follows from the nondegeneracy condition in \eqref{ndc} and the inclusions in \eqref{inc} and \eqref{nddu} that the BCQ condition  in \eqref{bcq} holds.
This tells us  that there is a neighborhood $\O$ of $\ou=\Phi(\ox)$ such that for any $u\in \O\cap \dom \vt$, we have 
$$
N_{\dom \vt} (\Xi(u))\cap \ker \nabla \Xi(u)^*=\{0\}.
$$
So, for any $k$ sufficiently large, we conclude from the chain rule for subdifferentials in \cite[Theorem~10.6]{rw}, $\Hat y^k\in  \partial g(\Phi(\ox))$, and \eqref{eb2} that there are   $\Hat \mu^k\in \sub \vt  (\Xi(\ou))$ and $\mu^k\in  \sub \vt  (\Xi (u^k))$ with $u^k=\Phi(\xk)-\bek$ such that
$$
\Hat y^k=\nabla \Xi(\ou)^*\Hat \mu^k\quad \mbox{and}\quad  y^k+\bek=\nabla \Xi(u^k)^*\mu^k.
$$
These allow us to obtain 
$$
\big\la \frac{y^k+\bek- \Hat y^k}{t_k}, \frac{u^k-\ou}{t_k}\big\ra=\big \la \frac{\nabla \Xi(\ou)^*(\mu^k- \Hat \mu^k)}{t_k}, \frac{u^k-\ou}{t_k}\big\ra+\big\la \frac{(\nabla \Xi(u^k)-\nabla \Xi(\ou))^*\mu^k}{t_k}, \frac{u^k-\ou}{t_k}\big\ra.
$$
The first term in the right-hand side of the above equality can be estimated by
\begin{align*}
\big \la \frac{\nabla \Xi(\ou)^*(\mu^k- \Hat \mu^k)}{t_k}, \frac{u^k-\ou}{t_k}\big\ra &= \big \la \frac{\mu^k- \Hat \mu^k}{t_k}, \frac{\nabla \Xi(\ou)(u^k-\ou)}{t_k}\big\ra\\
&= \big\la \frac{\mu^k- \Hat \mu^k}{t_k}, \frac{\Xi(u^k)-\Xi(\ou)}{t_k}\big\ra+ \big\la \frac{\mu^k- \Hat \mu^k}{t_k}, \frac{o(\|u^k-\ou\|)}{t_k}\big\ra\\
&\ge \big\la \frac{\mu^k- \Hat \mu^k}{t_k},\frac{o(t_k)}{t_k}\big\ra,
\end{align*}
where the last inequality results    from the monotonicity of the subgradient mapping $\sub \vt$ and the fact that $u^k-\ou=O(t_k)$. 
Also, we have 
$$
\big\la \frac{(\nabla \Xi(u^k)-\nabla \Xi(\ou))^*\mu^k}{t_k}, \frac{u^k-\ou}{t_k}\big\ra = \big\la \frac{(\nabla^2  \Xi(\ou)(u^k- \ou) +o(t_k))^*\mu^k}{t_k}, \frac{u^k-\ou}{t_k}\big\ra.
$$
Combining these results in the estimate 
\begin{equation}\label{eb08}
\big\la \frac{y^k+\bek- \Hat y^k}{t_k}, \frac{u^k-\ou}{t_k}\big\ra \ge \big\la \frac{\mu^k- \Hat \mu^k}{t_k},\frac{o(t_k)}{t_k}\big\ra+  \big\la \frac{(\nabla^2  \Xi(\ou)(u^k- \ou) +o(t_k))^*\mu^k}{t_k}, \frac{u^k-\ou}{t_k}\big\ra.
\end{equation}
By \eqref{eb4} and $\bek=o(t_k)$, we can conclude that the left-hand side in \eqref{eb08} converges to $\la \eta, \nabla \Phi(\ox)w\ra $. Moreover, a similar argument as the one after \eqref{ex5.3} and the fact that $y^k-\Hat y^k=O(t_k)$
tell us that $\mu^k- \Hat \mu^k=O(t_k)$. Thus, the first term on the right-hand side in \eqref{eb08} converges to $0$. Finally, the nondegeneracy condition in \eqref{ndc} and the inclusion in \eqref{inc} implies that 
the dual condition in \eqref{duq} is satisfied. By Theorem~\ref{calag}(d), we obtain that the sequence $\{\mu^k\b$   converges to $\omu$ as $k\to\infty$, where $\omu$ is the unique element in $M_{\ou,g}(\oy,0)$ with $M_{\ou,g}(\oy,0)$ taken from \eqref{mpg};
see Theorem~\ref{calag}.
Passing to the limit in \eqref{eb08} then shows that 
$$
\la \eta, \nabla \Phi(\ox)w\ra \ge  \big\la \bar\mu, \nabla^2  \Xi(\ou)(\nabla \Phi(\ox)w, \nabla \Phi(\ox)w)\big\ra.
$$
Recall that $\nabla\Phi(\ox)w \in K_g(\Phi(\ox),\oy)$, which together with \eqref{critc} brings us to $\nabla\Xi(\ou)(\nabla\Phi(\ox)w) \in K_\vt(\Xi(\ou),\omu)$.  
Employing now the established formula for the second subderivative of $g$ in Theorem~\ref{calag}(a) and using \eqref{ssvt} result in 
$$
 \big\la \bar\mu, \nabla^2  \Xi(\ou)(\nabla \Phi(\ox)w, \nabla \Phi(\ox)w)\big\ra=   \d^2 g(\ou, \oy) (\nabla \Phi(\ox)w).
$$   
Combining these two estimates proves \eqref{eb10} when (a) holds. 
By \eqref{eb10} and \eqref{eb14.5}, we obtain 
\begin{equation*}
\langle\nabla^2_{xx} L(\ox,\oy)w, w\rangle+\d^2g(\ou,\oy)(\nabla\Phi(\ox)w)\leq \langle\nabla^2_{xx}L(\ox, \oy)w, w\rangle+\langle\eta,\nabla\Phi(\ox)w\rangle\leq 0,
\end{equation*}
which contradicts the SOSC in \eqref{sosc1}, since  $w\ne 0$ and  $w\in \D$ with $\D$ taken from \eqref{sosc1}. This proves  the   estimate  in \eqref{ebb}.

To justify \eqref{eb}, we need to  show that shrinking the neighborhood $U$ from \eqref{ebb}, if necessary, we have 
\begin{equation}\label{dis}
\dist\big(y,M_{\ox,\psi}(0,0)\big)=O\big(r(x,y)\big)\quad\textrm{for all}\quad(x,y)\in U.
\end{equation}
To this end, pick $(x,y)$ satisfying \eqref{ebb} and set $z:={\rm prox}_g (\Phi(x)+y)-\Phi(x)$, which results in  $y-z\in \partial g(\Phi(x)+z)$ or $\Phi(x)+z\in (\sub g)^{-1}(y-z)$. It is easy to see that  $z\to 0$ as  $(x,y)\to(\ox,\oy)$. 
Thus, shrinking the neighborhood $U$ from \eqref{ebb}, if necessary, we can assume without loss of generality that $y-z\in V$, where $V$ is taken from \eqref{calmm}, and hence conclude 
from the estimate in \eqref{calmm} and the polyhedrality of $\Th$ that 
\begin{align*}
\dist(y&-z,M_{\ox,\psi}(0,0))\leq \ell\,\big(\dist(-\nabla_x L(\ox,y-z), N_\Th(\ox))+\dist(\Phi(\ox),(\partial g)^{-1}(y-z))\big)\\
&\leq\ell\, \big(\Vert\nabla_x L(x,y)-\nabla_xL(\ox, y-z) \Vert+\dist(-\nabla_xL(x, y); N_\Th(x))+\|\Phi(x)+z -\Phi(\ox)\|\big)\\
&\le O\big(\|x- \ox\|+\|z\|\big)+\ell\,\dist(-\nabla_xL(x, y), N_\Th(x))\\
&=O\big(r(x,y)\big),
\end{align*}
where the last equality comes from \eqref{ebb} and the definition of the residual function $r$. Since the distance function is Lipschitz continuous, we get
$$
\dist\big(y,M_{\ox,\psi}(0,0)\big)-\dist\big(y-z,M_{\ox,\psi}(0,0)\big)=O(\|z\|)=O(\|{\rm prox}_g (\Phi(x)+y)-\Phi(x)\|)=O\big(r(x,y)\big).
$$
Combining these estimates, we arrive at \eqref{dis}. This completes the proof. 
\end{proof}

The error bound in Theorem~\ref{error} was obtained previously in \cite[Theorem~5.9]{ms18} for ${\cal C}^2$-cone reducible constrained optimization problems
using a rather lengthy reduction process, a path that was not followed here. It was extended for CPLQ composite optimization problems in \cite[Theorem~3.6]{s20}.   Theorem~\ref{error} provides a general property -- condition (b) in this theorem-- under which 
the error bound can be achieved. The main question that remains open here is whether ${\cal C}^2$-decomposable functions satisfy in \eqref{ex5.2}. It is worth reiterating here that 
${\cal C}^2$-decomposable functions do not contain    CPLQ functions (see Remark~\ref{notdec}) and hence condition (b) in Theorem~\ref{error}  cannot be covered by condition (a) therein.

\section{Semi-Stability of Second Subderivatives}\label{sssub}
In parallel with the notion of semi-strict graphical derivative in Definition~\ref{ssgd}, we define in this section a semi-strict counterpart of the second subderivative of functions and use it to introduce a new regularity condition associated with the latter construction that plays a crucial role in our local convergence analysis of the ALM. 
\begin{Definition}\label{ssss} Assume that  $f\colon\X\to\oR$, $\ox\in \X$ with $f(\ox)$ finite, and $\ov\in \partial f(\ox)$.
\begin{itemize}[noitemsep,topsep=0pt]
\item [\rm (a)] 
The {semi-strict second subderivative} of $f$ at $\bar x$ for $\ov$ is defined by 
\begin{equation*}
\Hat\d^2f(\ox, \ov)(w) = \liminf_{\substack{
t\searrow 0, \, w'\to w \\
 v \to \ov,\;v\in \sub f(\ox)}} 
\Delta_t^2 f(\ox , v)(w'), \quad w\in \X.
\end{equation*}
\item[(b)] The second subderivative of $f$ at $\bar x$ for $\ov$  is said to be
 semi-stable if $\Hat \d^2f(\ox, \ov)(w)= \d^2f(\ox, \ov)(w)$ for any $w\in K_f(\ox,\ov)$. 
\end{itemize}
\end{Definition}


As demonstrated below, the semistability of second subderivatives is  satisfied for various classes of functions, important for modeling constrained and composite 
optimization problems. Note that the inequality  $\Hat \d^2f(\ox, \ov)(w)\le \d^2f(\ox, \ov)(w)$ always holds  for any $w\in \X$. Thus justifying semi-stability of second subderivatives reduces to investigating whether or not  the latter inequality holds as equality. 

\begin{Proposition} \label{domh}Assume that $f\colon\X\to\oR$, $\ox\in \X$ with $f(\ox)$ finite, and $\ov\in \Hat\sub f(\ox)$. Then 
we have 
$$
\dom \Hat\d^2f(\ox, \ov)\subset {K_f}(\ox,\bar v).
$$ 
\end{Proposition}
\begin{proof} 
Take $w\in \dom \Hat\d^2f(\ox, \ov)$ and conclude that $\Hat\d^2f(\ox, \ov)(w) <\infty$. Since we obtain for any $t>0$ that 
$$
\Delta_t^2 f(\bar x , v)(w')=\frac{\Delta_t f(\ox)(w')-\la v,w'\ra}{\sm t}\quad\mbox{with}\;\; \Delta_t f(\ox)(w'):=\frac{f(\ox+tw')-f(\ox)}{t},
$$
it follows from the definition of $\Hat\d^2f(\ox, \ov)(w)$ that 
$$
\al:=\liminf_{\substack{
  t\searrow 0, \, w'\to w \\
 v \to \ov,\;v\in \sub f(\ox)}} 
   \Delta_t f(\ox)(w')-\la v,w'\ra <\infty.
$$
If $\al>0$, we would get $\Hat \d^2f(\ox, \ov)(w)=\infty$, a contradiction. Thus, we must have $\al\le 0$, which 
 clearly tells us that $\d f(\ox)(w)\le \la \ov,w\ra$. Noting that the opposite inequality always holds because of  $\ov\in \Hat \sub f(\ox)$ (cf. \cite[Exercise~8.4]{rw}),   we arrive at the claimed inclusion. 
\end{proof}


Which classes of functions enjoy  the semi-stability of second subserivatives? To answer this question, we begin with the following observation. 
Suppose  that $f\colon\X\to\oR$ is  a polyhedral function,  meaning that  its epigraph is a polyhedral convex set in $\X\times \R$. Given $(\ox,\ov)\in \gph \sub f$,  it is known (cf. \cite[Proposition~13.9]{rw}) that 
$$
 \d^2f(\ox, \ov)=\dd_{{K_f}(\ox,\bar v)}.
$$
It follows from the convexity of $f$ that $\Delta_t^2 f(\bar x , v)(w')\ge 0$ for any $v\in \sub f(\ox)$, $w'\in \X$, and $t>0$. Thus, we get 
$\Hat \d^2f(\ox, \ov)(w)\ge 0$  for any $w\in \X$ and conclude from 
 the inequality $\Hat \d^2f(\ox, \ov)(w)\le \d^2f(\ox, \ov)(w)$, satisfied  for any $w\in \X$,
 that $\Hat \d^2f(\ox, \ov)(w)=0$ for any $w\in {K_f}(\ox,\bar v)$. Combining this 
  and Proposition~\ref{domh} confirms that 
\begin{equation}\label{sspl}
   \Hat\d^2f(\ox, \ov)=\dd_{{K_f}(\ox,\bar v)}= \d^2f(\ox, \ov),
\end{equation}
 which means that   the second subderivative of a polyhedral function is always semi-stable. 
 This motivates us to explore further whether or not a similar result holds for CPLQ functions.

 \begin{Example}\label{plqf}
Assume that the function $f : \X \to \oR $ with $\X=\R^n$ 
is CPLQ with the representation given in Example~\ref{ex:4}. 
It was proven in \cite[Propsoition~13.9]{rw}
that the second subderivative of $f$ at $\ox$ for $\ov\in \sub f(\ox)$ can be calculated by 
\begin{equation}\label{pwfor}
 \d^2 f(\ox , \ov ) (w)  
=\begin{cases}
\la A_i w, w\ra&\mbox{if}\;\; w\in K_{ C_i}(\ox,\ov_i),\\
\infty&\mbox{otherwise},
\end{cases}
\end{equation}
where $\ov_i:=\ov -A_i \ox -a_i$. We claim that  $\Hat \d^2f(\ox, \ov)(w)= \d^2f(\ox, \ov)(w)$ for any $w\in \X$, 
which proves that  the second subderivative of a   CPLQ function is always semi-stable. 
Take $w\in K_f(\ox,\ov)$ and select sequences
$t_k\searrow 0$, $v^k\to \ov$ with $v^k\in \sub f(\ox)$, and $w^k\to w$ such that 
$$
\Hat \d^2f(\ox, \ov)(w)= \lim_{k\to \infty}\Delta_{t_k}^2 f(\bar x , v^k)(w^k).
$$
Since $\Hat \d^2f(\ox, \ov)(w)\leq \d^2f(\ox, \ov)(w)<\infty$, we can assume without loss of generality that $\ox+t_kw^k\in \dom f= \cup_{i=1}^{s} C_i$ for any $k$.
Passing to a subsequence, if necessary, we find $j\in \{1,\ldots,s\}$ such that $\ox+t_kw^k\in C_j$ for all $k$ and therefore conclude that $w^k\in T_{C_j}(\ox)$ for all $k$. Clearly, we have $\ox \in C_j$.
Moreover, we derive from $w\in K_f(\ox,\ov)$ and \cite[Proposition~2.1(b)]{s20} that 
$$
\la A_j\ox+a_j,w\ra=\d f(\ox)(w)=\la \ov,w\ra,
$$
which in turn brings us  to $w\in K_{C_j}(\ox,\ov_j)$. 
On the other hand, it follows   from \cite[page~487]{rw} that 
\begin{equation}\label{sub}
\sub f(\ox)=\bigcap_{i\in I(\ox)} \big \{v\in \X |\, v-A_i \ox- a_i\in N_{ C_i}(\ox)\big \},
\end{equation}
where $I(\ox)=  \{i\in \{1,\ldots, s\}\, |\, \ox\in C_i  \}$. By $v^k\in \sub f(\ox)$ and $j\in I(\ox)$, we get $v^k-A_j \ox- a_j\in N_{ C_j}(\ox)$.
This, coupled with $w^k\in T_{C_j}(\ox)$, implies that 
$$
\la w^k,v^k-A_j \ox- a_j\ra\le 0 \quad \mbox{for all}\;\; k.
$$
Using this,  we conclude from  a simple calculation  that 
\begin{align*}
&\Delta_{t_k}^2 f(\bar x , v^k)(w^k)=\dfrac{f(\ox+t_kw^k)-f(\ox)-t_k\langle v^k,\,w^k\rangle}{\frac{1}{2}t_k^2}\\
&\hspace{.3in}=\dfrac{\frac{1}{2} \langle A_j (\ox+t_kw^k) ,\ox+t_kw^k \rangle + \langle a_j ,\ox+t_kw^k \rangle + \alpha_j - \frac{1}{2} \langle A_j \ox ,\ox \rangle - \langle a_j ,\ox \rangle - \alpha_j -t_k\langle v^k,\,w_k\rangle}{\frac{1}{2}t_k^2}\\
&\hspace{.3in}= \la A_j w^k, w^k\ra-\dfrac{\la w^k, v^k -A_j \ox -a_j\ra}{ \frac{1}{2}t_k}\ge \la A_j w^k, w^k\ra.
\end{align*}
This, combined with \eqref{pwfor} and $w\in K_{C_j}(\ox,\ov_j)$, 
 leads us to 
$$
\Hat \d^2f(\ox, \ov)(w)\ge  \la A_j w, w\ra=  \d^2 f(\ox , \ov ) (w).
$$
Since the opposite inequality always holds, we arrive at 
$$
\Hat \d^2f(\ox, \ov)(w) =  \d^2 f(\ox , \ov ) (w) \quad \mbox{for all}\;\; w\in  K_f(\ox,\ov).
$$
If $w\notin  K_f(\ox,\ov)$, we get from Proposition~\ref{domh} and \eqref{pwfor} that $\Hat \d^2f(\ox, \ov)(w)  = \d^2 f(\ox , \ov ) (w) =\infty$, which  proves our claim.
\end{Example}

The next  example reveals that the semi-stability of second subderivatives also holds for $\C^2$-decomposable functions.

\begin{Example}{\rm Assume that $g:\Y\to \oR$ is ${\cal C}^2$-decomposable at $\ou\in \Y$ with representation \eqref{decom} and that $\oy\in \partial g(\ou)$ and $\omu\in M_{\ou,g}(\oy,0)$. 
Assume further that the dual condition in \eqref{duq} is satisfied. 
 We claim that  $\Hat \d^2 g(\ou, \oy)(w) =  \d^2 g(\ou , \oy )(w)$ for any $w\in \Y$. To justify it, observe from   $\omu\in M_{\ou,g}(\oy,0)$ and \eqref{mpg} that $\oy=\nabla \Xi(\ou)^*\omu$ with $\omu\in \sub \vt(\Xi(\ou))$. It follows from the 
convexity of the sublinear function $\vt:\Z\to \oR$ that 
$$
0\le \Hat \d^2\vt(\Xi(\ou), \omu)(\xi)\le \d^2 \vt(\Xi(\ou), \omu)(\xi)=\dd_{K_\vt(\Xi(\ou), \omu)}(\xi) \quad \mbox{for all}\;\; \xi\in \Z,
$$
where the last equality comes from \eqref{ssvt}. This   implies that $\Hat \d^2\vt(\Xi(\ou), \omu)(\xi)= \d^2 \vt(\Xi(\ou), \omu)(\xi)$ for any $\xi\in K_\vt(\Xi(\ou), \omu)$.
If $\xi\notin K_\vt(\Xi(\ou), \omu)$, we infer from \eqref{ssvt} and  Proposition~\ref{domh}  that $\Hat \d^2\vt(\Xi(\ou), \omu)(\xi)= \d^2 \vt(\Xi(\ou), \omu)(\xi)=\infty$. These prove that  
 the second subderivative of $\vt$ at $\Xi(\ou)$ for $\omu$ is semi-stable. 
To obtain the same conclusion for $g$, we first observe, for any $w\in \Y$, that
\begin{align*}
\Hat\d^2 g(\ou, \oy)(w)&\le   \d^2 g(\ou, \oy)(w)\\
&=  \la \omu, \nabla^2\Xi(\ou)(w, w)\ra +\d^2\vartheta(\Xi(\ou), \omu)(\nabla \Xi(\ou)w)\\
&=  \la \omu, \nabla^2\Xi(\ou)(w, w)\ra +\Hat \d^2\vartheta(\Xi(\ou), \omu)(\nabla \Xi(\ou)w),
\end{align*}
where the penultimate step results from Theorem~\ref{dess}(a). 
Take $w\in \Y$ and select $t_k \searrow 0$, $y^k\to \oy$ with $y^k\in \sub g(\ou)$, and $w^k\to w$ such that
$
\Hat\d^2 g(\ou, \oy)(w) =\lim_{k\to\infty}\Delta_{t_k}^2 g(\ou, y^k)(w^k).
$
We know from the dual condition \eqref{duq} and \eqref{inc} that the BCQ condition in \eqref{bcq} holds.
Employing the chain rule for subdifferentials from \cite[Theorem~10.6]{rw},  we find $\mu^k\in \sub \vt(\Xi(\ou))$ such that $y^k=\nabla \Xi(\ou)^*\mu^k$. We can also conclude from  Theorem~\ref{calag}(d) that $\mu^k\to \omu$ as $y\to \oy$. 
Set $\xi^k:=\Xi(\ou+t_kw^k)/{t_k}$ and conclude from   \eqref{decom} that  
\begin{align*}
\Delta^2_{t_k} g(\ou,y^k)(w^k)&=\dfrac{\vt\big(\Xi(\ou+t_kw^k)\big)-\vt\big(\Xi(\ou)\big)-t_k\langle\nabla \Xi(\ou)^*\mu^k,w^k\rangle}{\frac{1}{2}t_k^2}\\
&= \dfrac{\vt\big(\Xi(\ou)+t_k\xi^k\big)-\vt\big(\Xi(\ou)\big)-t_k\langle\mu^k,\xi^k\rangle}{\frac{1}{2}t_k^2}+ \dfrac{\langle\mu^k,\Xi(\ou+t_kw^k)-t_k\nabla \Xi(\ou)w^k\rangle}{\frac{1}{2}t_k^2}\\
&= \Delta^2_{t_k} \vt(\Xi(\ou),\mu^k)(\xi^k)+ \la \mu^k,\nabla^2 \Xi(\ou)(w^k,w^k)\ra+ \dfrac{o(t_k^2)}{t_k^2}.
\end{align*}
Passing to the limit as $k\to\infty$ and using the facts that $\mu^k \to \bar\mu$ and $\xi^k \to \nabla\Xi(\ou)w$ lead us to 
$$
\Hat\d^2 g(\ou, \oy)(w) \ge \Hat \d^2\vt(\Xi(\ou), \omu)(\nabla \Xi(\ou)w) + \la \omu,\nabla^2 \Xi(\ou)(w,w)\ra.
$$
Combining the above estimates tells us that $\Hat\d^2 g(\ou, \oy)(w)=  \d^2 g(\ou, \oy)(w)$ for any $w\in \X$, which  proves   our claim. 
}
\end{Example}

We are now in a position to establish our main result of this section, a uniform quadratic growth condition
for the augmented Lagrangian in \eqref{aug} when the SOSC \eqref{sosc1} is satisfied. To provide a motivation for this result, we should remind 
the readers that it was proven in \cite[Theorem~4.1]{HaS21} that for the composite optimization problem in \eqref{comp} with $g$ therein being CPLQ, 
 the SOSC is equivalent to the quadratic growth condition for the augmented Lagrangian function. While the framework in  \cite{HaS21} requires that $g$ be CPLQ,
 it is not hard to see from the proofs therein that what actually is required is twice epi-differentiability   of the convex function $g$ in \eqref{comp}, a property that 
 is   satisfied for numerous classes of functions and sets; see \cite{mms,ms20,rw}. Below, we record this characterization of 
 the quadratic growth condition for the augmented Lagrangian function via the SOSC. Since its proof can be gleaned from  \cite[Theorem~4.1]{HaS21}  and since it will not be exploited 
 in this paper, we do not supply it with a proof. 
 \begin{Proposition}\label{quad} Let $(\ox,\oy)$ be a solution to the KKT system \eqref{vs}. Then the following properties are equivalent:
 \begin{enumerate}[noitemsep,topsep=0pt]
 \item  the SOSC in \eqref{sosc1} holds at $(\ox,\oy)$;
 \item there exist positive constants $\bar\rho$, $\gg$, and $\ell$ such that for any $\rho\ge\bar \rho$, the quadratic growth condition  
 $$
 \L(x,\oy,\rho)\ge \L(\ox,\oy,\rho)+\ell\| x-\bar x\|^2\quad \mbox{ for all }\;x\in\B_\gg (\ox)\cap \Th
 $$
holds.
\end{enumerate}
 \end{Proposition}
 
 While the characterization above seems appealing, we will see in the next section when establishing the error bound for consecutive terms of the ALM that 
 the quadratic growth condition in Proposition~\ref{quad}(b) is not sufficient. In fact, what is needed is a stronger version of this condition,  called   the {\em uniform} quadratic growth condition,
 which allows for the Lagrangian multiplier $\oy$ to move slightly and  the same  quadratic growth condition still holds. To achieve it, we begin with 
the calculation of the second-order difference quotient of the augmented Lagrangian \eqref{aug}. 
Note that $\ov$ in the second-order difference quotient of   $f$ in  \eqref{sodq} is taken from $\sub f(\ox)$. When $f$ is ${\cal C}^1$, it is known that
$\sub f(\ox)=\{\nabla f(\ox)\}$. In such a case, we drop $\ov$ from the notation of the second-order difference quotient of   $f$  in  \eqref{sodq} and simply write $\Delta_{t}^2 f(\ox)(w)$. 
The second-order difference quotient of the Moreau envelope was calculated  in the proof of \cite[Theorem~6.1]{pr} (see equation (6.8) in \cite{pr}).
For the augmented Lagrangian \eqref{aug}, a similar calculation was done in the proof of \cite[Lemma~3.2]{ks15}. Since the set $\Th$ in \eqref{comp} was not considered  in \cite{ks15},
we provide a proof for the readers' convenience.

\begin{Lemma}\label{augd} Let $(\ox,\oy)$ be a solution to the KKT system \eqref{vs} and set $\ov := \nabla_x\L(\ox, \oy, \rho)$ and 
$\Delta_{t}^2 \L(\ox,\oy,\rho)(w):={(\L(\ox+tw,\oy,\rho)- \L(\ox,\oy,\rho)-\la \ov,w\ra)}/{\sm t^2}$ for $\rho>0$.
Then for any $\rho>0$, $t>0$, and $w\in \X$,  we have 
\begin{eqnarray*}
\Delta_{t}^2 \L(\ox,\oy,\rho)(w)=\Delta_{t}^2 \ph(\bar x)(w) + \Delta_{t}^2 \la \oy,  \Phi\ra (\bar x)(w) + \inf_{u\in \Y}\big\{ \Delta_{t}^2 g(\Phi(\bar x) ,  \oy)(u)+\rho \| \Delta_t \Phi(\ox)(w)-u\|^2\big\},
\end{eqnarray*}
where $\Delta_{t}^2 \la \oy,  \Phi\ra (\bar x)$ stands for the second-order difference quotient of the mapping $x\mapsto \la \oy,\Phi(x)\ra$ at $\ox$ and 
where $\Delta_t\Phi(\ox)(w):=(\Phi(\ox+tw)-\Phi(\ox))/t$ for any $t>0$.
\end{Lemma}

\begin{proof} 
 We can conclude from   Proposition~\ref{fopag}(a)-(b) that $\ov=\nabla \ph(\ox)+\nabla \Phi(\ox)^*\oy$ and 
\begin{align*}
&\tfrac{2}{t^2}\big(\L(\ox+tw,\oy,\rho)- \L(\ox,\oy,\rho)-t\la \ov,w\ra\big)\\
=&\, \tfrac{2}{t^2}\big(\ph(\ox+tw)-\ph(\ox) + e_{ {1}/{\rho}} g \big(\Phi(\ox+tw)+\rho^{-1}\oy\big)-g(\Phi(\ox))-\sm\rho^{-1}\|\oy\|^2 -t\la \ov,w\ra \big)\\ 
=&\;  \Delta_{t}^2 \ph(\bar x)(w) -\tfrac{2}{t}
 \la \nabla\Phi(\ox)^*\oy, w\ra \\
&+\tfrac{2}{t^2}\big(\inf_{z\in \Y}\left\{\right. g(z)-g(\Phi(\ox)) - \la \oy, z-\Phi(\ox)\ra +\tfrac{1}{2}\rho\|z-\Phi(\ox) - (\Phi(\ox+tw) - \Phi(\ox))\|^2\\
&+\tfrac{1}{2}\rho\|z- \Phi(\ox+tw)-\rho^{-1}\oy\|^2-\tfrac{1}{2}\rho\|z-\Phi(\ox+tw)\|^2- \tfrac{1}{2}\rho^{-1}\|\oy\|^2+ \la \oy, z-\Phi(\ox)\ra\left.\right\}\big)\\
=&\; \Delta_{t}^2 \ph(\bar x )(w) +\inf_{u\in \Y}\big\{ \Delta_t^2g(\Phi(\ox), \oy)(u)+\rho\|u-\Delta_t\Phi(\ox)(w)\|^2\big\}\\
&+ \tfrac{2}{t^2}\la \oy, \Phi(\ox+tw)-\Phi(\ox)-t\nabla \Phi(\ox)w\ra\\
=&\;  \Delta_{t}^2 \ph(\bar x)(w)+ \Delta_t^2\la\oy, \Phi\ra(\ox)(w)   +\inf_{u\in \Y}\big\{ \Delta_t^2g(\Phi(\ox), \oy)(u)+\rho\|u-\Delta_t\Phi(\ox)(w)\|^2\big\}.
\end{align*}
This clearly completes the proof.   
\end{proof}

The main consequence of the semi-stability of second subderivatives resides in the following observation in which we  ensure the uniform quadratic growth condition for the augmented Lagrangian function.

\begin{Theorem}\label{thm:uqgc} Let $(\ox,\oy)$ be a solution to the KKT system \eqref{vs}.
Assume that the second subderivative of $g$ at $\Phi(\ox)$ for $\oy$ is semi-stable and  that  there is a positive constant $\ell$ such that the  SOSC
\begin{equation}\label{sosc}
\langle\nabla_{xx}^2L(\bar x,\olm)w,w\rangle+\d^2g (\Phi(\bar x),\oy) (\nabla \Phi(\ox)w )\ge \ell\|w\|^2\quad \mbox{ for all }\; w\in\D
\end{equation}
holds, where $\D$ is taken from \eqref{sosc1}.  Then, there exist positive constants $\bar \rho$ and $\gamma$ for which the uniform quadratic growth condition 
\begin{equation}\label{uqgc}
\begin{cases}
\L(x,y,\rho)\ge  \L(\ox,y,\rho)+\frac{\kappa}{2}\|x-\ox\|^2\\
\mbox{for all}\;\; x\in \Theta\cap\B_\gamma(\ox), \; y\in M_{\ox,\psi}(0,0)\cap \B_\gamma(\oy), \; \rho\ge \bar\rho,
\end{cases}
\end{equation} 
is fulfilled for all $\kappa\in [0,\ell)$, where $M_{\ox,\psi}(0,0)$ is the set of Lagrange multipliers associated with $\ox$ and is taken from \eqref{mulmap}.
\end{Theorem}

\begin{proof} Suppose by contradiction that \eqref{uqgc} fails. Thus, there exist $x^k\to \ox$ with $\xk\in \Th$,  $y^k\to \oy$  with $y^k\in M_{\ox,\psi}(0,0)$, and $\rho_k\to \infty$ such that 
$$
\L(x^k,y^k,\rho_k)<  \L(\ox,y^k,\rho_k)+\frac{\kappa}{2}\|x^k-\ox\|^2,
$$
holds for some $\kappa\in [0,\ell)$. Set $t_k:=\|x^k-\ox\|$ and $w^k:=(x^k-\ox)/t_k$. Passing to a subsequence if necessary, we can assume that $w^k\to w$ for some $w\in \X$ with $\|w\|=1$. Substituting $x^k=\ox+t_kw^k$ into the latter estimate, we arrive at 
$$
\Delta_{t_k}^2 \L(\ox,y^k,\rho_k)(w^k)+\frac{2}{t_k}\la v^k, w^k\ra=\frac{\L(\ox+t_kw^k,y^k,\rho_k)- \L(\ox,y^k,\rho_k)}{\sm t_k^2} <\kappa.
$$
where $v^k := \nabla_x\L(\ox, y^k,\rho_k)$. It follows from $y^k\in M_{\ox,\psi}(0,0)$ that $(\ox,y^k)$ is a solution to the KKT system in \eqref{vs}.
 By Proposition~\ref{fopag}(b), we have $-v^k=-(\nabla\ph(\ox)+\nabla\Phi(\ox)^*y^k) \in N_\Th(\ox)$.
 Set $\al_k:=\Delta_{t_k}^2 \ph(\bar x)(w^k) + \Delta_{t_k}^2 \la y^k,  \Phi\ra (\bar x)(w^k)$  and observe that 
\begin{equation}\label{alkl}
\al_k\to \langle\nabla_{xx}^2L(\bar x,\oy)w,w\rangle\quad\textrm{ and }\quad v^k\to \nabla_xL(\ox,\oy)\quad \mbox{as}\;\; k\to \infty.
\end{equation}
 It follows from  Lemma~\ref{augd} that  
$$
\al_k+\frac{2}{t_k}\la v^k, w^k\ra+\inf_{u\in \Y}\big\{ \Delta_{t_k}^2 g(\Phi(\bar x) ,  y^k)(u)+\rho_k \|u- \Delta_{t_k} \Phi(\ox)(w^k)\|^2\big\} <\kappa.
$$
Since $g$ is convex and $y^k\in \sub g(\Phi(\ox))$, we get that $\Delta_{t_k}^2 g(\Phi(\bar x) ,  y^k)(\cdot)$ is nonnegative. So, for any $k\in \N$, we find $u^k\in \Y$ such that 
\begin{equation}\label{imo}
\al_k+\frac{2}{t_k}\la  v^k, w^k\ra+ \Delta_{t_k}^2 g(\Phi(\bar x) ,  y^k)(u^k)+\rho_k \|u^k- \Delta_{t_k} \Phi(\ox)(w^k)\|^2<\kappa +\frac{1}{k}.
\end{equation}
  It also follows from the convexity of  $\Th$ 
that $\tfrac{2}{t_k}\la v^k, w^k\ra = \Delta_{t_k}^2\delta_\Theta(\ox, -v^k)(w^k) \geq 0$. Using these and \eqref{imo}, we 
obtain 
$$
 \| \Delta_{t_k} \Phi(\ox)(w^k)-u^k\|^2\leq\frac{1}{\rok}\big(\kappa + \tfrac{1}{k} -\al_k\big),
$$
which clearly yields $ \Delta_{t_k} \Phi(\ox)(w^k)-u^k \to 0$ as $k\to \infty$. Combining this with $\Delta_{t_k} \Phi(\ox)(w^k)\to \nabla \Phi(\ox)w$, we arrive at  
$u^k\to \nabla \Phi(\ox)w$. Employing \eqref{imo} again, we conclude that
$$
\frac{2}{t_k}\big(\Delta_{t_k} g(\Phi(\ox))(u^k)-\la y^k,u^k\ra\big)
=\Delta_{t_k}^2 g(\Phi(\bar x) ,  y^k)(u^k) <\kappa+\frac{1}{k}-\al_k,
$$
which implies that 
$$
\liminf_{k \to \infty} \frac{2}{t_k}\big(\Delta_{t_k} g(\Phi(\ox))(u^k)-\la y^k,u^k\ra\big)<\infty.
$$
In particular, it holds that $\liminf_{k\to \infty} \Delta_{t_k} g(\Phi(\ox))(u^k)-\la y^k,u^k\ra \le 0$. This, coupled with $u^k\to \nabla \Phi(\ox)w$,
 confirms that $\d g(\Phi(\ox))(\nabla\Phi(\ox)w)\le \la \oy,\nabla\Phi(\ox)w\ra$. It follows from the convexity of  $g$  and $\oy\in \sub g(\Phi(\ox)) $ that 
 the opposite inequality $\d g(\Phi(\ox))(\nabla\Phi(\ox)w)\ge \la \oy,\nabla\Phi(\ox)w\ra$ always holds. Thus, we get $\nabla\Phi(\ox)w\in K_g(\Phi(\ox),\oy)$. Repeating the above arguments for
 the  second-order quotient of $\delta_\Theta$ at $\ox$ for $-v^k$ and $w^k$ yields $w\in K_\Th(\ox, -\nabla_xL(\ox,\oy))$. Thus, it results from the definition of $\D$ in \eqref{sosc1} that $w\in \D$. 
Observe again that \eqref{imo} gives us the estimate 
$$
\al_k+ \Delta_{t_k}^2 g(\Phi(\bar x) ,  y^k)(u^k) +\Delta_{t_k}^2\delta_\Theta(\ox, -v^k)(w^k)<\kappa+\frac{1}{k}.
$$
Passing to the limit, using \eqref{alkl}, and remembering that $u^k\to \nabla \Phi(\ox)w$  demonstrate that 
\begin{equation}\label{est11}
\langle\nabla_{xx}^2L(\bar x,\oy)w,w\rangle +\Hat \d^2g (\Phi(\bar x),\oy) (\nabla \Phi(\ox)w )+\Hat \d^2\delta_\Theta\big(\ox,-\nabla_xL(\ox,\oy)\big)(w) \le \kappa.
\end{equation}
Since $\Th$ is a polyhedral convex set, $\dd_\Th$ is CPLQ. By Example~\ref{plqf}, the second subderivative of  $\delta_\Theta$ is semi-stable
at $\ox$ for $-\nabla_xL(\ox,\oy)$. Moreover, \eqref{sspl} and $w\in K_\Th(\ox, -\nabla_xL(\ox,\oy))$ imply  that 
$$
\Hat \d^2\delta_\Theta(\ox,-\nabla_xL(\ox,\oy))(w)= \delta_{K_\Th(\ox, -\nabla_xL(\ox,\oy))}(w)=0,
$$
which, coupled with semi-stability of the second subderivative of  $g$  at $\Phi(\ox)$ for $\oy$, results in 
\begin{equation*}
\langle\nabla_{xx}^2L(\bar x,\oy)w,w\rangle +\d^2g (\Phi(\bar x),\oy) (\nabla \Phi(\ox)w )\leq \kappa,
\end{equation*} 
a contradiction with \eqref{sosc}, since  $\kappa\in [0,\ell)$,  $\|w\|=1$, and $w\in \D$. This completes the proof.
\end{proof}

It is worth adding here that the SOSCs in \eqref{sosc1} and \eqref{sosc} are equivalent. This, in fact, results from the lower semicontinuity of the second subderivative 
function $\d^2 g(\Phi(\ox),\oy)$ therein, which is due to \cite[Propsoition~13.5]{rw}. 
Note that the proof of Theorem~\ref{thm:uqgc} is inspired by that of \cite[Theorem~3.4(b)]{ss}, which was replicated in \cite[Proposition~3.3]{ks}. 

The uniform quadratic growth condition for augmented Lagrangian functions was first established in \cite[Proposition~3.1]{fs12} 
for NLPs without appealing to the concept of the second subderivative. It was generalized in \cite[Theorem~4.3]{HaS21} for 
the composite problem in \eqref{comp} with the modeling function $g$ therein being CPLQ via a different approach.
It  was recently generalized in \cite[Theorem~1]{ding} for SDPs when the multiplier $\oy$ is taken from the relative interior of the set of Lagrange multipliers. 
Theorem~\ref{thm:uqgc} goes far beyond these frameworks and provides a relatively easy proof of this important result, which is based on the new concept of semi-stability of second subderivatives, introduced 
in this section. We should also add here that the uniform quadratic growth condition in \eqref{uqgc} implies the validity of the SOSC in \eqref{sosc} for some $\ell>0$. This can be proven using Proposition~\ref{quad}.
Since such a result will not be used in our convergence analysis and since it requires calculation of the second subderivative of the augmented Lagrangian function, we did not state it in Theorem~\ref{thm:uqgc}.

\section{Local Convergence Analysis of ALM}\label{local}

After some preparations in Sections~\ref{calm} and \ref{sssub}, we are ready to analyze the local convergence of an inexact version of the ALM
described below. 
\begin{Algorithm}[ALM]\label{Alg1}
Choose $(x^0,y^0)\in \X \times \Y$, $\bar\rho>0$, and $\hat c>0$. Pick a sequence $\{\rok\c$ with  $\rok\ge \bar\rho$ for all $k$ and a function $\epsilon:\R_+\to\R_+$ satisfying $\epsilon(t) = o(t)$ and set $k:=0$. Then
\begin{itemize}[noitemsep,topsep=0pt]
\item [{\rm (1)}] if $(\xk,y^k)$ satisfies a suitable termination criterion, stop;
\item[{\rm (2)}]  otherwise, set $\epsilon_k:= \epsilon(r(\xk, \lmk))$ with the residual function $r$ taken from \eqref{res} and choose the primal-dual update  $(\xkk, \ykk)$ according to \eqref{xkk} and \eqref{kkt2} such that
\begin{equation}\label{est7}
\|\xkk-\xk\|+\|\ykk- y^k\|\leq \hat c\,r(\xk, y^k);
\end{equation}
\item[{\rm (3)}] set $k\leftarrow k+1$ and go to Step~1.
\end{itemize}
\end{Algorithm} 

The roadmap to ensure  the convergence of the sequence $\{(\xk,y^k)\b$, constructed by Algorithm~\ref{Alg1}, was already established in Theorem~\ref{thm:fischer}. To this end, 
we are going to demonstrate that   assumptions (a)-(c) in the latter result are satisfied under the same assumptions utilized   in Theorem~\ref{error}. To begin, 
define the solution mapping $\s:\X\times\Y \to \X\times \Y$ to the canonical perturbation of the generalized equation in \eqref{geco}   by
\begin{equation}\label{pkkt}
\s(v, w) := \big\{(x, y) \in \X\times\Y\, \big|\, (v, w) \in \Psi(x,y)+G(x, y)\big\},\;\; (v,w)\in \X\times \Y,
\end{equation}
where $\Psi$ and $G$ are taken from \eqref{ALMge}. Observe that assumption (a) in Theorem~\ref{thm:fischer} requires that $\s$ enjoy a   calmness property. 
Below, we show that a certain calmness property of $\s$ is equivalent to the error bound estimate in \eqref{eb}. 
\begin{Proposition}\label{casol}
Let $(\ox, \oy)$ be a solution to the KKT system in \eqref{vs}. Then, the following properties  are equivalent.
\begin{enumerate}[noitemsep,topsep=2pt]
\item There exist constants $\gamma>0$ and $\kappa\ge 0$ for which the error bound estimate in \eqref{eb} holds.
\item There are  positive constants   $\dd'$ and $\ell $   for which the solution mapping $\s$ from \eqref{pkkt} enjoys the  calmness property 
\begin{equation}\label{casol2}
\s(v, w)\cap \B_{\dd'}(\ox,\oy) \subset \big(\{\ox\}\times M_{\ox,\psi}(0,0)\big) +\ell\, \|(v,w)\|\B\quad \mbox{for all}\;\;  (v, w) \in \dd'\B,
\end{equation}
where   $M_{\ox,\psi}(0,0)$ is the set of Lagrange multipliers associated with $\ox$ and is defined by \eqref{mulmap}. 
\end{enumerate}
\end{Proposition}
\begin{proof} The equivalence of (a) and (b) was established in  \cite[Proposition~3.8]{s20}
for the composite problem in \eqref{comp} with $g$ therein being CPLQ.  A close look into its proof, however, tells us that the latter assumption on $g$ was  not exploited and the given argument works 
for any convex function $g$. 
\end{proof}

We now proceed with an elaboration of assumption (b) in Theorem~\ref{thm:fischer} for the inexact ALM from Algorithm~\ref{Alg1}, which 
  consists of two steps: (1) verifying the solvability of the subproblem in \eqref{subp} and (2) establishing an error bound estimate for consecutive iterates of the ALM.
  We begin with solvability of subproblems. This was already justified in \cite[Propsoition~5.2]{HaS21} when the convex function $g$ in \eqref{comp} is CPLQ.
  The given proof therein did not utilize such an assumption on $g$ and indeed works for any convex function $g$. Below, we record this result and skip its proof. 
\begin{Proposition}\label{solv} 
Assume that $(\ox,\oy)$ is a solution to the KKT system \eqref{vs} such that the  SOSC in \eqref{sosc} hold at $(\ox,\oy)$. 
Assume further that the second subderivative of $g$ at $\Phi(\ox)$ for $\oy$ be semi-stable. 
Take the positive constants $\gg$ and $\bar\rho$   from Theorem~{\rm\ref{thm:uqgc}}.
Then, there exist positive constants $\hat\ell$ and $\hat\gamma\in (0, \gg)$   such that for any $\rho\in [ \bar\rho,\infty)$,
 the optimal solution mapping $S_\rho:\Y\tto \X$, defined by 
\begin{equation}\label{eq177}
S_\rho(y):=\argmin\big\{\L(x,y,\rho)\;\big|\; x\in\Th\cap\B_{\hat\gamma}(\ox)\big\}, \;\;y\in \Y,
\end{equation}
enjoys the uniform  isolated calmness property 
\begin{equation}\label{pt9}
   S_\rho(y) \subset\{\ox\} +\hat\ell \Vert y-\oy\Vert\B
\end{equation}
 and satisfies the inclusion $\emp \neq S_\rho(y)\subset  \inte  \B_{\hat\gamma}(\ox)$ for all $y\in \B_{\hat\gamma/(2\hat\ell)}(\oy)$.  
 \end{Proposition}

Next, we address the second property, required in assumption (b) in Theorem~\ref{thm:fischer}, which is an error bound estimate for 
consecutive terms of our inexact ALM. While the proof of the result below mostly uses a similar argument in our recent work  
in \cite[Theorem~5.7]{HaS21}, it differs from the proof of the latter result in the second half part. In fact, some parts of the proof of \cite[Theorem~5.7]{HaS21}
depend heavily on that fact that $g$ from \eqref{comp} was assumed to be CPLQ. In those parts, we proceed with a new idea, used recently in \cite[Proposition~5]{ding}.

\begin{Proposition}\label{est} 
Assume that $(\ox,\oy)$ is a solution to the KKT system in  \eqref{vs} for which all the assumptions in Theorem~{\rm\ref{error}} hold and that the second subderivative of $g$ at $\Phi(\ox)$ for $\oy$ is semi-stable. 
Take $\bar\rho$   from Theorem~{\rm\ref{thm:uqgc}} and $\hat\gg$   from Proposition~{\rm\ref{solv}}. 
Then, there exist positive constants   $\hat\varepsilon$ and $\hat c$ such that for any $\rho\geq \bar\rho$, any $(x,y)\in\B_{\hat\varepsilon}(\ox,\oy)$ satisfying   $r(x,y)>0$, and any   optimal solution $\z$ to the regularized problem 
\begin{equation}\label{regp}
\mini \L(u, y, \rho)\quad \mbox{subject to}\quad u\in \Th\cap \B_{\hat\gg}(\ox)
\end{equation} 
 the error bound estimate
\begin{equation}\label{est0}
\Vert \z-x\Vert+ \Vert y_s-y \Vert\le\hat c\,r(x,y)\quad \mbox{with}\;\;y_s:=y+\rho\nabla_y \L(s, y, \rho)
\end{equation}
holds,  where the residual function  $r$ is  defined by \eqref{res}.
\end{Proposition}

\begin{proof} Since the SOSC in \eqref{sosc1} holds, it follows from lower semicontinuity of $\d^2g (\Phi(\bar x),\oy)$ therein (cf. \cite[Proposition~13.5]{rw}) that there is a positive constant $\ell$
for which the SOSC in \eqref{sosc} is satisfied. By Theorem~\ref{thm:uqgc}, we find    positive  constants $\bar\rho$, $\gamma$, and $\kappa$ for which 
the uniform quadratic  growth condition in \eqref{uqgc} holds  for all $x\in \Th\cap\B_\gamma(\ox)$, all $y\in M_{\ox,\psi}(0,0)\cap\B_\gamma(\oy)$, and all $\rho\ge\bar\rho$. 
According to Proposition~\ref{solv}, the solution mapping  $S_\rho$ enjoys  the uniform isolated calmness property in \eqref{pt9} and   $S_\rho(\ty)\subset  \inte  \B_{\hat\gamma}(\ox)$ for 
all $\ty\in \B_{\hat\gamma/(2\hat\ell)}(\oy)$ and all $\rho\ge \bar\rho$, where both positive constants $\hat\ell$ and $\hat\gg$ are taken from this proposition. 
Thus,  for every $\ty\in \B_{\hat\gamma/(2\hat\ell)}(\oy)$ and every $\rho\ge \bar\rho$, 
any optimal solution $\z$ to \eqref{regp} satisfies the first-order optimality condition 
\begin{equation}\label{optc}
0\in \nabla_x\L(\z, \ty, \rho)+ N_\Th(\z).
\end{equation}
To justify \eqref{est0},  assume by contradiction that it does not hold. Thus, we find  a sequence $\{(\xk, y^k, \rok)\b\subset \X \times \Y\times [\bar\rho, \infty)$ 
with $(\xk,y^k) \to( \ox,\oy)$   and an optimal solution $\z^k$  to \eqref{regp} associated with $(y, \rho)=(y^k, \rok)$ such that
\begin{equation}\label{est1}
\Vert \z^k-\xk\Vert+\| q^k-y^k\|>k\, r_k\quad \mbox{with}\;\;q^k:=y^k+\rok\nabla_y \L(s^k, y^k, \rok),
\end{equation} 
 where $r_k:=r(\xk, y^k)>0$. The latter particularly tells us that $r_k$ is finite, which  yields $\xk\in \Th$ for all $k$ due to \eqref{res}.  
Denoting by $\beta_k$ the left-hand side of \eqref{est1}, we get    $r_k=o(\beta_k)$.
The definition of the residual function $r$ in \eqref{res} then leads us to 
\begin{equation}\label{est2}
-\nabla_xL(\xk, y^k)+o(\xik)\in N_\Th(\xk)\quad\textrm{and}\quad \Phi(\xk)+o(\xik)= {\rm{prox}}_g\big(\Phi (\xk)+y^k\big).
\end{equation}
Using the definition of $\beta_k$ and passing to a subsequence, if necessary, we can find $(\zeta,\eta)\in \X\times \Y$ such that 
\begin{equation}\label{est4}
\dfrac{\z^k-\xk}{\xik}\to \zeta\quad \mbox{and}\quad  \dfrac{ q^k-y^k}{\xik} \to  \eta \quad \mbox{with}\;\; (\zeta,\eta)\neq 0.
\end{equation}
Since the set of Lagrange multiplier $M_{\ox,\psi}(0,0)$ from \eqref{mulmap} is convex and closed, $P_{M_{\ox,\psi}(0,0))}(y^k)$  exists and is a singleton. Set $\ty^k:=P_{M_{\ox,\psi}(0,0))}(y^k)$.
By Theorem~\ref{error},  the estimate in \eqref{eb} is satisfied, which allows us to conclude that $\xk-\ox=O(r_k)$ and $y^k-\ty^k=O(r_k)$  for all $k$ sufficiently large. Thus, we have 
\begin{equation}\label{est3}
\xk-\ox=o(\xik)\quad\textrm{and}\quad y^k-\ty^k=o(\xik)\quad\textrm{as } k\to\infty,
\end{equation}
which together  with $y^k\to \oy$ yields $\ty^k\to \oy$  and hence $\ty^k\in M_{\ox,\psi}(0,0)\cap\B_\gamma(\oy)$ for all $k $ sufficiently large. 
Combining the latter with  $\z^k\in \Th\cap \B_\gamma(\ox)$, $\rho_k\ge\bar\rho$, and \eqref{uqgc} brings us to 
\begin{align} 
\Vert \z^k-\ox\Vert^2&\le\frac{2}{\kappa}\big (\L(\z^k, \ty^k, \rok) - \L(\ox, \ty^k, \rok)\big)\nonumber\\
&\le\frac{2}{\kappa}\big(\L(\z^k, y^k, \rok)+\la \nabla_y\L(\z^k, y^k, \rok), \ty^k-y^k\ra - \L(\ox, \ty^k, \rok)\big),\nonumber
\end{align}
where   the last  inequality results from the fact that the mapping $y\mapsto \L(\z^k, y, \rok) $ is ${\cal C}^1$ and concave; see \cite[Exercise~11.56]{rw}. 
Since $\z^k\in S_\rok(y^k)$, we can conclude from the definition of the augmented Lagrangian function in \eqref{aug} 
that 
$$
\L(\z^k, y^k, \rok)\le \L(\ox, y^k, \rok) \le \ph(\ox)+ g(\Phi(\ox))= \L(\ox, \ty^k, \rok),
$$
where the equality comes from Proposition~\ref{fopag}(a) and $\ty^k\in M_{\ox,\psi}(0,0)$. Combining the last two estimates and using then \eqref{est1} lead us to 
\begin{equation}\label{est5}
\Vert \z^k-\ox\Vert^2 \le  \frac{2}{\kappa\rho_k} \langle q^k-y^k,\ty^k-y^k \rangle\le\dfrac{2}{\kappa\rok}\| q^k-y^k\|\cdot\Vert\ty^k-y^k\Vert.
\end{equation}
{\bf Claim I}. We have $\z^k-\ox=o(\beta_k)$, $\rok\|\z^k-\ox\|^2=o(\beta_k)$, and $\rok\|\z^k-\ox\|^2=o(\beta_k^2)$.

To prove the first estimate, we use \eqref{est5} together with \eqref{est4}-\eqref{est3} and $\rok\ge \bar\rho$ to obtain
$$
\dfrac{\Vert \z^k-\ox\Vert^2}{\xik^2}\leq\dfrac{2}{\kappa\rok} \dfrac{\| q^k-y^k\|}{\xik} \cdot\frac{\Vert \ty^k-y^k\Vert}{\xik}\le  \dfrac{2}{\kappa\bar \rho} \dfrac{\| q^k-y^k\|}{\xik} \cdot\frac{\Vert \ty^k-y^k\Vert}{\xik} \to 0.
$$
The remaining estimates can be verified by a similar argument via \eqref{est4}-\eqref{est5}.  \\
{\bf Claim II}. We have $s^k\to \ox$ as $k\to \infty$. 

To prove this claim, we infer from  $\ty^k\in M_{\ox,\psi}(0,0)$ that  $\ty^k \in \partial g(\Phi(\ox))$, and hence  $\Phi(\ox) =\prox\big(\Phi(\ox)+\rho^{-1}\ty^k\big)$ for any $\rho>0$ due to  \eqref{prm}. 
 Thus, we derive from  \eqref{est1}, \eqref{est5}, \eqref{ueq19}, and  $\rok\ge \bar\rho$ that 
\begin{align}
\Vert \z^k-\ox\Vert^2 & \le\dfrac{2}{\kappa }\|  \Phi(\z^k)-\proxk\big(\Phi(\z^k)+\rho_k^{-1}y^k\big)\|\cdot\Vert\ty^k-y^k\Vert\nonumber\\
&=\frac{2}{\kappa}\big\|\Phi(\z^k) -\Phi(\ox) +\proxk\big(\Phi(\ox)+\rho_k^{-1}\ty^k\big)-\proxk\big(\Phi(\z^k)+\rho_k^{-1}y^k\big)\big\| \cdot\Vert\ty^k-y^k\Vert \nonumber\\
&\leq \frac{2}{\kappa}\big(2\|\Phi(s^k)-\Phi(\ox)\|+(\bar\rho)^{-1}\|\ty^k-y^k\|\big)\Vert\ty^k-y^k\Vert\nonumber\\
&=\big(O (\|\z^k-\ox\| )+ \tfrac{2}{\kappa}(\bar\rho)^{-1}\Vert \ty^k-y^k\Vert \big)\Vert\ty^k-y^k\Vert, \label{sk0}
\end{align} 
 where the inequality is due to the fact that the proximal mapping $\prox$ is nonexpansive (cf. \cite[Proposition~12.19]{rw}). 
 Since $\|s^k-\ox\|$ is bounded and  $\|y^k-\ty^k\|\to 0$, we deduce from \eqref{sk0} that $\Vert \z^k-\ox\Vert\to 0$, which proves the claim.
 
 To proceed, we use the first estimate in Claim I and \eqref{est3} to get
\begin{equation*}
\zeta=\lim_{k\to\infty}\dfrac{\z^k-\xk}{\xik}=\lim_{k\to\infty}\dfrac{s^k-\ox}{\xik}-\lim_{k\to\infty}\dfrac{\xk-\ox}{\xik}=0-0=0.
\end{equation*}
Our goal is to show that $\eta=0$, which together with $\zeta=0$ leads us to a contradiction with \eqref{est4}. To this end, we proceed with considering two cases. Assume first that 
  either $\{\rok\b$ or $\{{\rok}/{\xik}\b$ is bounded. 
Using a similar argument as the one for \eqref{sk0}, we obtain via \eqref{est1} and \eqref{ueq19} that 
\begin{align*}
 \dfrac{ \| q^k-y^k\|}{\xik} =\frac{\rok}{\xik} \big\Vert \Phi(\z^k)-\proxk\big(\Phi(\z^k)+\rho_k^{-1} y^k\big)\big\Vert \le O\big(\frac{\rok}{\xik}\|\z^k-\ox\|\big)+\dfrac{\Vert \ty^k-y^k\Vert }{\xik}.  
\end{align*} 
It also follows  from $s^k-\ox=o(\beta_k)$ that $ {\rok}\Vert \z^k-\ox\Vert / \xik \to0$. Using this and \eqref{est3} and passing to the limit in the inequality above, we get  $\eta =0$, a contradiction with \eqref{est4}.

Assume now that both sequences $\{\rok\b$ and  $\{{\rok}/{\xik}\b$  are unbounded.  We can assume by passing  to a subsequence if necessary that 
\begin{equation}\label{unb}
\rok\to \infty\quad \mbox{and}\quad \frac{\rok}{\beta_k}\to \infty\;\;\mbox{as}\;\; k\to \infty.
\end{equation}
Since $\z^k$ is an optimal solution to \eqref{regp} associated with $(y^k,\rok)$, we deduce from \eqref{optc} that 
\begin{equation*}
0\in \nabla_x\L(\z^k, y^k, \rok)+ N_\Th(\z^k)= \nabla \varphi(\z^k) +\nabla \Phi(\z^k)^*q^k+N_\Th(\z^k),
\end{equation*}
 where the equality comes from \eqref{ueq19} and  the definition of $q^k$ from \eqref{est1}. By the definition of the normal cone, we have
\begin{align}
0\le&\; \la \nabla \varphi(\z^k) +\nabla \Phi(\z^k)^*q^k, \ox-s^k\ra =\la \nabla \varphi(\z^k), \ox-s^k\ra-\la q^k,\nabla \Phi(\z^k)( s^k-\ox)\ra\nonumber\\
=&\;\la \nabla \varphi(\z^k), \ox-s^k\ra- \la q^k,\nabla\Phi(\ox)(s^k-\ox)\ra
-\la q^k,(\nabla\Phi(s^k)-\nabla\Phi(\ox))(s^k-\ox)\ra.
\label{nc1}
\end{align}
Recall also that $\ty^k\in M_{\ox,\psi}(0,0)$. By \eqref{mulmap}, the latter implies that  $\ty^k\in \sub g(\Phi(\ox))$ and 
\begin{equation*}\label{k2}
0\in \nabla_x L(\ox,\ty^k)+N_\Th(\ox)=\nabla \varphi(\ox) +\nabla \Phi(\ox)^*\ty^k +N_\Th(\ox).
\end{equation*}
Again, the definition of the normal cone leads us to 
\begin{align}
0&\le  \la \nabla \varphi(\ox) +\nabla \Phi(\ox)^*\ty^k, s^k-\ox\ra = \la \nabla \varphi(\ox), s^k-\ox\ra+\la \ty^k,\nabla \Phi(\ox)( s^k-\ox)\ra.
\label{nc2}
\end{align}
Adding both sides of \eqref{nc1} and \eqref{nc2} brings us to 
 \begin{align}
 \la q^k-\ty^k , \nabla\Phi(\ox)(s^k-\ox) \ra &\le 
 - \la \nabla \varphi(s^k)-\nabla \varphi(\ox), s^k-\ox\ra -\la q^k,(\nabla\Phi(s^k)-\nabla\Phi(\ox))(s^k-\ox)\ra\nonumber\\
 &\le  (1+ \|q^k\|)O(\|s^k-\ox\|^2).
  \label{nc3}
 \end{align}
 On the other hand, it follows from \eqref{kkt2}  and  the definition of $q^k$ in \eqref{est1} that  $q^k \in \partial g(z^k)$ where $z^k:= \Phi(\z^k)- \rho_k^{-1}(q^k-y^k)$. Recall that $\ty^k\in \sub g(\Phi(\ox))$. Thus, we get from the monotonicity of $\partial g$ that
\begin{align}
 0\leq&\; \la q^k-\ty^k , \Phi(s^k)- \rho_k^{-1}(q^k-y^k)-\Phi(\ox) \ra\nonumber\\
  =&\;  \la q^k-\ty^k ,\nabla\Phi(\ox)(s^k - \ox) +O(\|s^k - \ox\|^2) -\rho_k^{-1}(q^k - y^k)\ra\nonumber\\
 \leq &\; \la q^k-\ty^k ,\nabla\Phi(\ox)(s^k - \ox)\ra -\rho_k^{-1}\|q^k - y^k\|^2\nonumber\\
 &+ \|q^k - \ty^k\|O(\|s^k - \ox\|^2)+ \rok^{-1}\|y^k - \ty^k\|\cdot\|q^k-y^k\|.\label{nc4}
  \end{align}
Combining \eqref{nc3} and \eqref{nc4} brings us to 
\begin{align}
  \|q^k-y^k\|^2
  \le&\;  \rok(1+\|q^k\|  +\|q^k-\ty^k\|)O(\|s^k-\ox\|^2)+ \|y^k - \ty^k\|\cdot\|q^k-y^k\|\nonumber\\
  \le&\; (1+\|y^k\|)O(\rok\|s^k-\ox\|^2)+(2\|q^k-y^k\|+\|y^k-\ty^k\|)O(\rok\|s^k-\ox\|^2)\nonumber\\
  &+\|y^k - \ty^k\|\cdot\|q^k-y^k\|\nonumber\\
  =&\; o(\beta_k^2)+O(\beta_k)o(\beta_k)+ o(\beta_k)O(\beta_k) = o(\beta_k^2),\nonumber
  \end{align}
  where the last estimate comes from Claim I, \eqref{est4}, \eqref{est3}, and boundedness of the sequence $\{y^k\}$. Dividing both sides by $\beta_k^2$ and passing then to the limit tell us via \eqref{est4} that $\|\eta\|^2=0$, implying that 
  $\eta=0$. This is a contradiction with $(\zeta,\eta)\neq (0,0)$ and thus completes the proof.
\end{proof} 

To finish our discussion about assumption (b) in Theorem~\ref{thm:fischer}, we should point out that one more step is required to be taken.
Comparing the right-hand side of \eqref{est0} with the  estimate in assumption (b) indicates that the residual function $r$ in \eqref{est0}  should be replaced with the distance function to the solution set of the KKT system in \eqref{vs}.
That requires to assume that $\Th$ in \eqref{comp}   be an affine set, as recorded below.

\begin{Proposition}  \label{erct}  Assume that $(\ox,\oy)$ is a solution to the KKT system \eqref{vs} with $\Th$ therein being an affine set. Then 
there exist a positive constant   $\kappa$ and a neighborhood $U$ of $(\ox,\oy)$ such that
\begin{equation}\label{eb1}
r(x, \lambda) \leq \kappa\big(\|x-\ox\|+\dist(y, M_{\ox,\psi}(0,0)))\big)\quad\textrm{for all }\;\; (x, y)\in U,
\end{equation} 
where $M_{\ox,\psi}(0,0)$ is defined in \eqref{mulmap}.
\end{Proposition}
\begin{proof} The claimed estimate  was recently justified for the KKT system in \eqref{vs} with $g$ being CPLQ. Its proof, however, did not use the latter condition on $g$
and works for any convex function $g$. Thus, we omit the proof and refer the readers to the latter result.
\end{proof}

What remains is to show that assumption (c) of Theorem~\ref{thm:fischer} holds automatically for the inexact ALM proposed in Algorithm~\ref{Alg1}. That will be done inside the proof of the next result  in which we establish  the wellposedness and local convergence of our inexact ALM. 

\begin{Theorem}\label{lcon}
Assume that  $(\ox, \oy)$ is a solution to the KKT system in \eqref{vs} for which all the assumptions in Theorem~{\rm\ref{error}} hold and that $\Th$ in \eqref{comp} is an affine set. 
Assume further that the second subderivative of $g$ at $\Phi(\ox)$ for $\oy$ is semi-stable. 
Then, there exist positive constants $\hat c,\, \bar\rho$, and $\varepsilon_0$ such that for any starting point $(x^0, y^0)\in \B_{\varepsilon_0}(\ox, \oy) $ and any sequence $\{\rok\}_{k\in \N}$ with $\rok\geq \bar\rho$, there is a primal-dual sequence $\{(\xk, y^k)\}_{k\in\N}$ satisfying \eqref{xkk} with $\ek = o(r(\xk, y^k))$ and the estimate \eqref{est7}, where $r$ is the residual function defined by \eqref{res}. Moreover, any such a sequence converges to $(\ox, \widehat y)$ for some $\widehat y\in M_{\ox,\psi}(0,0)$, and the rates of convergence of $\{(\xk, y^k)\b$ to $(\ox, \widehat y)$ and of $\{\dist\big((\xk, y^k), \{\ox\}\times M_{\ox,\psi}(0,0)\big)\b$ to zero are Q-linear. Furthermore, if $\rok\to\infty$, the rates of convergence of both sequences are Q-superlinear.
\end{Theorem}

\begin{proof} To justify the claims, we need to show that assumptions (a)-(c) in Theorem~\ref{thm:fischer} are satisfied. We begin with assumption (a), namely the calmness of the solution mapping $\s$ from \eqref{pkkt}
at $(0,0)$ for $(\ox,\oy)$. To achieve it, observe from Theorem~\ref{error} that the error-bound estimate in \eqref{eb} holds. Take the neighborhood $U$ of $(\ox,\oy)$ from \eqref{eb} and choose $\dd>0$ such that $\B_\dd(\ox,\oy)\subset U$.
By \eqref{eb}, we get 
\begin{equation}\label{fsol}
\s(0,0)\cap \B_\dd(\ox,\oy)=\big(\{\ox\}\times M_{\ox,\psi}(0,0)\big)\cap \B_\dd(\ox,\oy).
\end{equation}
 Moreover, it follows from Proposition~\ref{casol}
that there are positive constants $\dd'$ and $\ell$ for which \eqref{casol2} holds. Shrinking $\dd'$ if necessary, we can assume that $\max\{\dd',\ell\dd'\}\le \dd/2$. Thus, it results from \eqref{casol2} that
\begin{align*}
\s(v, w)\cap \B_{\dd'}(\ox,\oy) &\subset \Big(\big(\{\ox\}\times M_{\ox,\psi}(0,0)\big) +\ell\, \|(v,w)\|\B\Big)\cap \B_{\dd/2}(\ox,\oy)\\
&\subset (\big(\{\ox\}\times M_{\ox,\psi}(0,0)\big) \cap  \B_{\dd}(\ox,\oy)+\ell\, \|(v,w)\|\B\\
&= \s(0,0)\cap \B_\dd(\ox,\oy)+ \ell\, \|(v,w)\|\B\subset   \s(0,0)+ \ell\, \|(v,w)\|\B,
\end{align*}
for all $(v,w)\in \dd'\B$, which confirms the calmness of  $\s$   at $(0,0)$ for $(\ox,\oy)$.

We turn next  to assumption (b) in Theorem~\ref{thm:fischer}. Take the positive constants $\hat \ve$ and $\hat c$ from Proposition~\ref{est} and assume without loss of generality that 
$\B_{\hat \ve}(\ox,\oy)\subset U$ and $\hat \ve\le 2\dd$, where $U$ is taken from \eqref{eb1} and $\dd$ comes from \eqref{fsol}. Combining now the estimates in \eqref{est0} and \eqref{eb1} implies that 
\begin{align*}
\Vert \z-x\Vert+ \Vert y_s-y \Vert\le\hat c\,r(x,y) &\le \hat c \kappa\big(\|x-\ox\|+\dist(y, M_{\ox,\psi}(0,0))\big)\\
&\le 2\hat c \kappa\dist((x,y), \{\ox\}\times  M_{\ox,\psi}(0,0))
\end{align*}
for all $(x,y)\in \B_{\hat \ve}(\ox,\oy)$ such that $r(x,y)>0$. If  $(x,y)\in \B_{\hat \ve}(\ox,\oy)$ and $r(x,y)=0$, one can   set $s:=x$, $y_s:=y$ and observes that $(s,y_s)$ is a solution to the KKT system in \eqref{vs}.
It then follows from Proposition~\ref{fopag}(b) that $s$ and $y_s$  satisfy  \eqref{xkk} and \eqref{kkt2}, respectively. Combining this with \eqref{fsol} allows us to conclude  that for any $(x,y)\in \B_{\hat \ve}(\ox,\oy)$
there is a pair $(s,y_s)$ satisfying  \eqref{xkk},  \eqref{kkt2},  and the estimate 
$$
\Vert \z-x\Vert+ \Vert y_s-y \Vert\le 2\hat c \kappa\dist((x,y), \s(0,0)),
$$
which confirms the validity of assumption (b) in Theorem~\ref{thm:fischer}.

Finally, to justify assumption (c) in Theorem~\ref{thm:fischer}, pick the positive constant  $\bar \rho$ from  Proposition~\ref{est},  take any $(\tx, \ty)\in \B_{\hat\ve}(\ox, \oy)$, any $(x, y)\in \X\times \Y$ 
with $\|(x, y)-(\tx, \ty)\|\leq  \hat c\,\dist((x,y), \s(0,0))$, and any $\rho\ge \bar \rho$. Choose also the tolerance parameter $\epsilon$ such that $\epsilon=o(r(x, y))$. For any  
 $(w_1, w_2) \in \Psi(x, y) - \A(x, y, \tx, \ty, \epsilon, \rho)$, where $\Psi$ and $\A$ are taken from  \eqref{ALMge} and \eqref{A},  
we find $b\in \B$ such that  
\begin{align*} 
\|(w_1, w_2)\|&=\|(\epsilon b, \rho^{-1}(y -\ty)\| \leq   \epsilon + \rho^{-1}\|y -\ty\|\\
&\le  o(r(x, y))+  \rho^{-1}\hat c\,\dist((x,y), \s(0,0))\\
&=\frac{o(r(x, y))}{r(x,y)} r(x,y)+ \rho^{-1}\hat c\,\dist((x,y), \s(0,0))\\
&\le \frac{o(r(x, y))}{r(x,y)}  \kappa\big(\|x-\ox\|+\dist(y, M_{\ox,\psi}(0,0))\big)   + \rho^{-1}\hat c\,\dist((x,y), \s(0,0))\\
&\le \frac{o(r(x, y))}{r(x,y)}  2  \kappa\big(\dist((x,y), \{\ox\}\times  M_{\ox,\psi}(0,0))\big)   + \rho^{-1}\hat c\,\dist((x,y), \s(0,0))\\
&=\big( \frac{2  \kappa \, o(r(x, y))}{r(x,y)} + \rho^{-1}\hat c\big)\dist((x,y), \s(0,0)),
\end{align*}
where the second inequality results from \eqref{eb1} and the last equality comes from \eqref{fsol}. Defining the function $\omega:\X\times \Y\times \X\times \Y\times (0,\infty)\to \R_+$
by $\omega(x,y,\tx,\ty,\rho)=  {2  \kappa \, o(r(x, y))}/{r(x,y)} + \rho^{-1}\hat c$, we obtain the second estimate in assumption (c) in Theorem~\ref{thm:fischer}. Since $\omega(x,y,\tx,\ty,\rho)\to 0$
as $(x,y,\tx,\ty,\rho)\to (\ox,\oy,\ox,\oy,\infty)$, the first estimate in assumption (c) also holds. Appealing now to the latter theorem proves all of convergence claims and hence completes the proof.
\end{proof}

We close this section by commenting on the established local convergence in Theorem~\ref{lcon}. For the case of NLPs, this result reduces to \cite[Theorem~3.4]{fs12}, where 
such a result was achieved for the first time for NLPs. Note that the calmness of multiplier mapping, assumed in Theorem~\ref{lcon},   is automatically satisfied by Proposition~\ref{cmpr} for NLPs.
When  g in \eqref{comp} is CPLQ, Theorem~\ref{lcon} covers our recent result in \cite[Theorem~5.7]{HaS21}. When $g=\dd_{\S^n_+}$ and $\Th=\S^n$ in \eqref{comp}, which allow us to cover SDPs, it was recently shown in \cite[Theorem~2]{ding} that if the multiplier $\oy$ is taken from the relative interior of the set of Lagrange multipliers  and the SOSC is satisfied, the inexact version 
of the ALM is Q-linearly convergent. Using a stronger version of the SOSC, the authors in \cite{ding2} obtained the R-linear convergence for   the primal sequence and Q-linear convergence for 
the dual sequence of the ALM without assuming the strict complementarity condition.  Theorem~\ref{lcon} improves all  these results for nonlinear SDPs  by demonstrating that the SOSC alone 
suffices to achieve the Q-linear convergence of the primal-dual sequence for the inexact ALM. Finally, we should add that local convergence of the ALM was established recently in \cite[Theorem~3.1]{r22}
under the strong variational convexity, a condition that is equivalent to the strong second-order condition for NLPs, for the composite optimization problem in \eqref{comp} with $\Th=\X$. First, our analysis 
uses a weaker version of the SOSC. Indeed, the strong version of the SOSC amounts to assuming the classical SOSC in a neighborhood of a solution to the KKT system, which is much stronger than the version 
we used in this paper. Second, while we establish Q-linear convergence of the primal-dual sequence from the ALM, the latter result in \cite{r22} only presents the R-linear convergence of the primal sequence 
and Q-linear convergence of the dual sequence from this method. It is not hard to see that our  primal-dual Q-linear convergence implies the R-linear convergence of the primal and dual sequences in the ALM.

%

\section{Appendix} \label{appe}

In this section, we aim to justify two independent results from our local convergence analysis of the ALM, which  play an important role in the process of justifying some results in Sections~\ref{calm} and \ref{sssub}.
We begin with a  useful characterization of  the {\em outer Lipschitzian} property of the   multiplier mapping $M_{\ou,g}$ from \eqref{mpg}  via the dual  condition \eqref{duq}.

\begin{Proposition}\label{calag} 
Assume that $g:\Y\to \oR$ is ${\cal C}^2$-decomposable at $\ou\in \Y$ with representation \eqref{decom} and that $\oy\in \partial g(\ou)$ and $\omu\in M_{\ou,g}(\oy,0)$,
where  $M_{\ou,g}$ is defined by \eqref{mpg}. Then the following properties are equivalent.
\begin{enumerate}[noitemsep,topsep=2pt]
\item The set of Lagrange multipliers $M_{\ou, g}(\oy, 0)$ is a singleton and there exist constants $\ell\ge 0$ and $\ve>0$ ensuring the error bound
estimate 
\begin{equation*}\label{gf01}
\| \mu-\omu\| \le\ell\big(\|\nabla \Xi(\ou)^*\mu -\oy\|+\dist\big(\Xi(\ou),(\sub \vartheta)^{-1}(\mu)\big)\big)\;\mbox{ for all }\; \mu\in\B_\ve(\omu).
\end{equation*}

\item The dual  condition \eqref{duq}  is satisfied.

\item  There exist positive numbers $\gamma$ and $\kappa$ such that  the inclusion 
\begin{equation*} 
M_{\ou,g}(y,w)\cap \B_\gamma(\omu) \subset\{\omu\}+\kappa(\Vert y-\oy\Vert+\Vert w\Vert)\B\;\mbox{ for all  }\;(y,w)\in\B_{\gamma}(\oy,0)
\end{equation*}
holds.

\item  There exist positive numbers $\gamma$ and $\kappa$ such that  the inclusion 
\begin{equation*}\label{pt1}
M_{\ou,g}(y,w)  \subset\{\omu\}+\kappa(\Vert y-\oy\Vert+\Vert w\Vert)\B\;\mbox{ for all  }\;(y,w)\in\B_{\gamma}(\oy,0)
\end{equation*}
holds.
\end{enumerate}
\end{Proposition} 
\begin{proof} The equivalence of (a), (b), and (c) was established in \cite[Theorem~4.1]{ms18} when the convex function $\vartheta$ was the indicator function of a convex set. 
A closer look into the proof of the latter result shows that a similar argument works for $\vartheta$. Clearly, we have the implication (d)$\implies$(c). Assume now that (b) holds. To prove
(d), suppose by contradiction that there exist  sequences  $(y^k,w^k)\to(\oy,0)$ as $k\to\infty$ and   the corresponding multipliers $\mu^k\in M_{\ou,g}(y^k,w^k)$ satisfying the inequality
\begin{equation}\label{pt2}
\Vert \mu^k-\omu\Vert>k(\Vert y^k-\oy\Vert+\Vert w^k\Vert)\;\mbox{ for all }\;k\in\N.
\end{equation}
Note that  the equivalence of (a) and (b) yields $M_{\ou,g}(\oy,0)= \{\omu\}$. Set $t_k:=\Vert \mu^k-\omu\Vert$. We claim that the sequence $\{\mu^k\b$ is bounded. 
Indeed, it follows from \cite[page~138]{mi} that $\vartheta$ is twice epi-differentiable at $\Xi(\ox)$ for $\omu$   and its second subderivative has a representation in the form 
\begin{equation}\label{ssvt}
\d^2\vartheta(\Xi(\ou), \omu)=\dd_{K_{\vartheta}(\Xi(\ou), \omu)}.
\end{equation}
 Employing \cite[Theorem~13.40]{rw}, we arrive at 
 \begin{equation}\label{ssub0}
D (\sub \vartheta) (\Xi(\ou),\omu)(w)=\sub \big(\sm \d^2\vartheta(\Xi(\ou), \omu)\big)(w) \quad \mbox{for all}\;\; w\in K_{\vartheta}(\Xi(\ou), \omu).
\end{equation}
In particular, this tells us that 
\begin{equation}\label{ssub}
D (\sub \vartheta) (\Xi(\ou),\omu)(0)=\sub \big(\sm \d^2\vartheta(\Xi(\ou), \omu)\big)(0)= N_{K_{\vartheta}(\Xi(\ou), \omu)}(0)= K_{\vartheta}(\Xi(\ou), \omu)^*.
\end{equation}
Take $w\in K_{\vartheta}(\Xi(\ou), \omu)$ and observe that $w\in \dom \d\vt(\Xi(\ou))$. Since the inclusion $ \dom \d\vt(\Xi(\ou))\subset T_{\dom \vt}(\Xi(\ou))$ always holds, we arrive at 
$K_{\vartheta}(\Xi(\ou), \omu)\subset T_{\dom \vt}(\Xi(\ou))$. This, coupled with  \eqref{ssub}, leads us to  
\begin{equation}\label{inc}
N_{\dom \vt} (\Xi(\ou))\subset K_{\vartheta}(\Xi(\ou), \omu)^*= D(\partial \vartheta)(\Xi(\ou), \omu)(0),
\end{equation}
 which implies  via \eqref{duq} that the BCQ in \eqref{bcq} holds. To justify the boundedness of $\{\mu^k\b$, assume by contradiction that it is unbounded.
In this case, we can pass to a subsequence if necessary to ensure that the sequence $\{\mu^k/\|\mu^k\|\b$ converges to some $\xi$ that $\|\xi\|=1$. Since $\mu^k\in M_{\ou,g}(y^k,w^k)$, it 
follows from \eqref{mpg} that $ \nabla\Xi(\ou)^*\mu^k=y^k$, which in turn implies that  $\xi\in \ker \nabla\Xi(\ou)^*$. Moreover, we have  $\mu^k\in \sub \vt(\Xi(\ou)+w^k)\subset \partial \vt(0)$, by Proposition~\ref{sulin}(b).
Take $w\in \dom \vt$ and conclude via the definition of the subdifferential of convex functions that
$$\la \mu^k, w\ra \leq \vt(w) - \vt(0) = \vt(w).$$
Dividing both sides by $\|\mu^k\|$ and passing to the limit, we obtain $\la \mu, w\ra\le 0$ for any $w\in \dom \vt$. This clearly tells us that $\xi\in N_{\dom \vt}(0)$, which together with $\xi\in \ker \nabla\Xi(\ou)^*$ contradicts \eqref{bcq}
and hence proves the boundedness of $\{\mu^k\b$. Recall that $M_{\ou,g}(\oy,0)=\{\omu\}$.
Since $\{\mu^k\b$ is bounded, we can assume that it is convergent by passing to a subsequence if necessary. Thus, it follows from $\mu^k\in M_{\ou,g}(y^k,w^k)$ and $M_{\ou,g}(\oy,0)=\{\omu\}$ that $\mu^k\to \omu$.
This immediately  implies that $t_k=\|\mu^k-\omu\|\to 0$. 
Observe then from \eqref{pt2} that $\|y^k-\oy\|= o(t_k)$ and $\|w^k\| = o(t_k)$.
Set $\eta^k: =(\mu^k-\omu)/t_k$ and assume without loss of generality that $\eta^k\to \eta$  with  $\|\eta\|=1$. 
Thus, we have
$\eta^k  \in (\sub \vt(\Xi(\ou)+t_kw^k/t_k) - \omu)/{t_k}$.
Passing to the limit as $k\to\infty$ leads us to $\eta\in D(\sub \vt)(\Xi(\ou),\omu)(0)$, due to the fact that $w^k=o(t_k)$.
Moreover, we have $\nabla\Xi(\ou)^*(\mu^k-\omu)=y^k-\oy$, which together with $y^k-\oy=o(t_k)$ results in $\eta\in \ker \nabla\Xi(\ou)^*$. The latter together contradict 
\eqref{duq}, since $\eta\neq 0$, and thus proves (d). 
\end{proof}

Recall from \cite[Definition~13.65]{rw} that a proper function  $f:\X \to \oR$ is parabolically regular at $\ox\in \dom f$ for  $\ov\in \sub f(\ox)$  if for any $w$ such that  $\d^2 f(\bar x , \ov)(w)<\infty  $, there exist, among the sequences $t_k \searrow 0$ and $w_k\to w$ with 
$\Delta_{t_k}^2 f(\bar x , \ov)(w_k) \to \d^2 f(\bar x , \ov)(w)$, those with the additional property that 
$
\limsup_{k\to \infty}  {\|w_k-w\|}/{t_k}<\infty.
$
Parabolic regularity was recently studied extensively in \cite{mms,ms20} for different classes of sets and functions. Below, we show that ${\cal C}^2$-decomposable  functions
enjoy this property when the dual condition in \eqref{duq} is satisfied.  To this end, we need to recall the concept of parabolically epi-differentiability of functions from \cite[Definition~13.59]{rw}. Given     $f:\X \to \oR$,  $\ox\in \X$ with $f(\ox)$ finite, and $w\in \X$ with $\d f(\ox)(w)$ finite, the {\em parabolic subderivative} of $f$ at $\ox$ for $w$ with respect to $z$ is defined by 
\begin{equation*}\label{lk02}
\d^2 f(\bar x)(w\verl z)= \liminf_{\substack{
   t\searrow 0 \\
  z'\to z
  }} \Delta^2_t f(\bar x)(w\verl z'),
\end{equation*}
where $\Delta^2_t f(\bar x)(w\verl z'):= (f(\ox+tw+\frac{1}{2}t^2 z')-f(\ox)-t\d f(\ox)(w))/{\frac{1}{2}t^2}$. 
The function  $f$ is called {\em parabolically epi-differentiable} at $\ox$ for $w$ if  
$$\dom \d^2 f(\ox)(w \verl \cdot)=\big\{z\in \X|\,  \d^2 f(\ox)(w \verl z)<\infty \big\}\neq \emptyset,$$
 and for every $z \in \X$ and every sequence $t_k\searrow 0$  there exists a  sequences $z^k\to z$ such that $\Delta^2_{t_k} f(\bar x)(w\verl z^k) \to \d^2 f(\bar x)(w\verl z)$.

\begin{Theorem} \label{dess}Assume that $g:\Y\to \oR$ is ${\cal C}^2$-decomposable at $\ou\in \Y$ with representation \eqref{decom} and that $\oy\in \partial g(\ou)$, $\omu\in M_{\ou,g}(\oy,0)$,
where  $M_{\ou,g}$ is defined by \eqref{mpg},  and  the dual condition in \eqref{duq}
is satisfied. 
Then, the following properties hold.
\begin{enumerate}[noitemsep,topsep=2pt]
\item The function $g$ is parabolically regular  at $\ou$ for $\oy$ and its  
second subderivative can be calculated by 
\begin{equation*}\label{ex3.2}
\d^2g(\ou, \oy)(w) = \la \omu, \nabla^2\Xi(\ou)(w, w)\ra +\d^2\vartheta(\Xi(\ou), \omu)(\nabla \Xi(\ou)w)\quad\textrm{ for all} \;\; w \in \Y.
\end{equation*}
\item The mapping $\sub g$ is proto-differentiable at $\ou$ for $\oy$ and its proto-derivative can be calculated by 
$$
D(\sub g)(\ou,\oy)(w)= \nabla^2 \la \omu, \Xi\ra (\ox) w+ \nabla \Xi(\ou)^*D (\sub \vartheta) (\Xi(\ou),\oy)(\nabla \Xi(\ou)w)\quad\textrm{ for all} \;\; w \in \Y.
$$
\end{enumerate}

\end{Theorem}
\begin{proof} To prove (a), we first show that $\vt$ is  parabolically regular  at $\Xi(\ou)=0$ for $\omu$, where $\vt$ and $\Xi$ are taken from \eqref{decom}. To this end, 
take $w\in \Z$ such that $\d^2\vartheta(\Xi(\ou), \omu)(w)<\infty$, which via \eqref{ssvt} reads as $w\in K_{\vartheta}(\Xi(\ou), \omu)$. By the sublinearity of $\vt$, we can conclude from the latter inclusion that $ \la \omu, w\ra=\d \vt(\Xi(\ou))(w)=\vt(w)$. This yields
\begin{equation*} 
\lim_{t\searrow 0} \Delta^2_t \vartheta(\Xi(\ou), \omu) (w)= \lim_{t\searrow 0} \frac{t(\vt(w) - \la \omu, w\ra)}{\tfrac{1}{2}t^2} =0=\d^2\vartheta(\Xi(\ou), \omu) (w).
\end{equation*}
Take $t_k \searrow 0$ and set  $w^k:= w$ for any $k\in \N$. According to the above relationships, we have $\Delta^2_{t_k} \vartheta(\Xi(\ou), \omu) (w^k)\to \d^2\vartheta(\Xi(\ou), \omu) (w)$, 
which proves that    $\vt$ is  parabolically regular  at $\Xi(\ou)$ for $\omu$. We are going next to show $\vt$ is  parabolically epi-differentiable at $\Xi(\ou)$ for any $w\in K_{\vt}(\Xi(\ou),\omu)$.
To achieve it, take  $w\in K_{\vt}(\Xi(\ou),\omu)$ and conclude for any $z\in \Z$ that 
\begin{equation}\label{pedvt}
\Delta^2_t \vt(\Xi(\ou))(w \verl z) = \frac{\vt(tw+\tfrac{1}{2}t^2z) - t\d \vt(\Xi(\ou))(w)}{\tfrac{1}{2}t^2} = \frac{\vt(w+\tfrac{1}{2}t z)-\vt(w)}{\tfrac{1}{2}t} = \Delta_{t/2} \vt(w)(z).
\end{equation}
Since $\vt$ is proper and convex, it follows from \cite[Example~7.27]{rw} that it is always epi-differentiable in the sense of \cite[Definition~7.23]{rw}. Thus,   for any $z\in \Z$  and $t_k \searrow 0$, 
there exists $z^k\to z$ such that $ \Delta_{t_k} \vt(w)(z^k)\to \d \vt(w)(z)$. Combining this with \eqref{pedvt} tells us that  for any $z\in \Z$  and $t_k \searrow 0$, we can find 
$z^k\to z$ such that $\Delta^2_{t_k} \vt(\Xi(\ou))(w \verl z^k)\to \d^2 \vt(\Xi(\ou))(w \verl z)$.  Moreover, it is not hard to see that $\d \vt(w)(0)<\infty$. This, coupled with \eqref{pedvt},
leads us to $0\in \dom \d^2 \vt(\Xi(\ou))(w \verl \cdot)$, since 
$$
\dom \d^2 \vt(\Xi(\ou))(w \verl \cdot)=\big\{z\in \Z|\;   \d^2 \vt(\Xi(\ou))(w \verl z)<\infty\big\}=\big\{z\in \Z|\;   \d \vt(w)(z)<\infty\big\}.
$$
These demonstrate that $\vt$ is  parabolically epi-differentiable at $\Xi(\ou)$ for $w\in K_{\vt}(\Xi(\ou),\omu)$. 

From the proof of Proposition~\ref{calag}, we know that the dual condition in \eqref{duq} implies the BCQ condition in \eqref{bcq} and uniqueness of the Lagrange multiplier $\omu$ in $M_{\ou,g}(\oy,0)$. Appealing now to \cite[Theorem~5.4 and Remark~5.3]{ms20}, we obtain both assertions in (a). 
To prove (b), 
observe   from \eqref{ssvt} that $\dom \d^2\vartheta(\Xi(\ou), \omu)=K_{\vt}(\Xi(\ou),\omu)$ and from \eqref{ssub} that the dual condition in \eqref{duq}
can be equivalently expressed as 
$$
N_{K_{\vt}(\Xi(\ou),\omu)}(0)\cap \ker \nabla \Xi(\ou)^*=\{0\}. 
$$
Thus, it follows from the chain rule for subdifferentials from \cite[Theorem~10.6]{rw} that, for any $w\in \Y$ such that $\nabla \Xi(\ou)w\in K_{\vt}(\Xi(\ou),\omu)$, we have 
\begin{align*}
\sub_w\big(\sm\d^2\vartheta(\Xi(\ou), \omu)\big)(\nabla \Xi(\ou)w)&= \nabla \Xi(\ou)^* \sub\big(\sm\d^2\vartheta(\Xi(\ou), \omu)\big)(\nabla \Xi(\ou)w)\nonumber\\
&=\nabla \Xi(\ou)^*D (\sub \vartheta) (\Xi(\ou),\omu)(\nabla \Xi(\ou)w),
\end{align*}
where the last equality comes from \eqref{ssub0}. This, coupled with  \cite[Corollary~3.9]{ms20} and the formula for $\d^2g(\ou, \oy)$ from (a), implies that $\sub g$ is proto-differentiable at $\ou$ for $\oy$ and that 
\begin{align*}
D(\sub g)(\ou,\oy)(w)&= \sub\big(\sm \d^2g(\ou, \oy)\big)(w)=  \sm \nabla_w \la \omu, \nabla^2\Xi(\ou)(w, w)\ra +\sub_w\big(\sm\d^2\vartheta(\Xi(\ou), \omu)\big)(\nabla \Xi(\ou)w)\\
&= \nabla^2 \la \omu, \Xi\ra (\ox) w+ \nabla \Xi(\ou)^*D (\sub \vartheta) (\Xi(\ou),\oy)(\nabla \Xi(\ou)w),
\end{align*}
for all $w\in \Y$, which proves (b) and hence ends the proof.
\end{proof}

\end{document}